\documentclass[10pt,reqno]{amsart}
\usepackage{amsmath,amsfonts,amsthm,amssymb}
\usepackage{datetime}
\usepackage[shortlabels]{enumitem}
\usepackage{tikz}
\usepackage{soul}
\usepackage{hyperref}

\newcommand{\R}{\mathbb{R}}
\newcommand{\N}{\mathbb{N}}
\newcommand{\id}{\operatorname{id}}

\newcommand{\inter}{\operatorname{int}}

\newcommand{\epi}{\operatorname{epi}}

\newcommand{\inner}[1]{\left\langle #1 \right\rangle}
\newcommand{\Hp}{\mathcal{H}^+_\infty(f)}
\newcommand{\G}{\mathcal{G}}
\newcommand{\Hm}{\mathcal{H}^-_0(f)}
\newcommand{\HmT}{\mathcal{H}^-_\infty(T)}
\newcommand{\HpT}{\mathcal{H}^+_\infty(T)}
\newcommand{\reach}{\operatorname{reach}}
\newcommand{\spn}{\operatorname{span}}
\newcommand{\supp}{\operatorname{supp}}
\newcommand{\mylambda}{t}

\newcommand{\one}{e_N} %\newcommand{\one}{\mathbf{1}}

\newcommand{\bs}{\setminus}

\theoremstyle{plain}
\newtheorem{theorem}{Theorem}[section]
\newtheorem{lemma}[theorem]{Lemma}

\newtheorem{corollary}[theorem]{Corollary}
\newtheorem{proposition}[theorem]{Proposition}

\theoremstyle{remark}
\newtheorem{remark}[theorem]{Remark}

\theoremstyle{definition}
\newtheorem{definition}[theorem]{Definition}
\newtheorem{example}[theorem]{Example}

\numberwithin{equation}{section}

\begin{document}
\title[Nonlinear Perron-Frobenius theory]{A Unified Approach to Nonlinear Perron-Frobenius Theory}
%\title[Nonlinear Perron-Frobenius theory]{Nonlinear Perron-Frobenius Theory and Collatz-Wielandt numbers}
\author[B. Lins]{Brian Lins}
\date{}
\address{Brian Lins, Hampden-Sydney College}
\email{blins@hsc.edu}

\begin{abstract}
Let $f:\R^n_{> 0} \rightarrow \R^n_{>0}$ be an order-preserving and homogeneous function.  We show that the set of eigenvectors of $f$ in $\R^n_{>0}$ is nonempty and bounded in Hilbert's projective metric if and only if $f$ satisfies a condition involving upper and lower Collatz-Wielandt numbers of readily computed auxiliary functions. This condition generalizes a test for the existence of eigenvectors using hypergraphs that was proved by Akian, Gaubert, and Hochart.  We include several examples to show how the new condition can be combined with the hypergraph test to give a systematic approach to determine when homogeneous and order-preserving functions have eigenvectors in $\R^n_{>0}$.  We also observe that if the entries of $f$ are real analytic functions on $\R^n_{>0}$, then the set of eigenvectors of $f$ in $\R^n_{>0}$ is nonempty and bounded in Hilbert's projective metric if and only if the eigenvector is unique, up to scaling.  %Finally, we give new conditions for the normalized iterates of an order-preserving homogeneous function on $\R^n$ to converge to an eigenvector and we give complete necessary and sufficient conditions for existence and uniqueness of entrywise positive eigenvectors of functions that are order-preserving, homogeneous, multiplicatively convex, and real analytic.  
\end{abstract}

\maketitle
 
\section{Introduction} \label{sec:intro}

The classical Perron-Frobenius theorem guarantees that an irreducible nonnegative matrix always has a unique (up to scaling) entrywise positive eigenvector with eigenvalue equal to the spectral radius of the matrix.  There is a long history of nonlinear extensions of the Perron-Frobenius theorem to functions defined on a cone that are homogeneous (of degree one) and preserve the partial order induced by the cone. Functions defined on the positive orthant $\R^n_{> 0}$ are of particular interest in applications. Here we use $\R^n_{>0}$ to denote the set of vectors in $\R^n$ with strictly positive entries and $\R^n_{\ge 0}$ denotes the vectors with nonnegative entries. For $x, y \in \R^n$, we say that $x \le y$ when $y-x \in \R^n_{\ge 0}$. By \emph{order-preserving}, we mean that if $x,y$ are in the domain of $f$ and $x \ge y$, then $f(x) \ge f(y)$; and \emph{homogeneous} means that $f(t x) = t f(x)$ for all $t > 0$. 
%Although an order-preserving homogeneous function $f: \R^n_{\ge 0} \rightarrow \R^n_{\ge 0}$ always has an eigenvector with nonnegative entries, a more important and delicate question is whether or not $f$ has an eigenvector with all positive entries. Thus, we focus on order-preserving homogeneous functions $f: \R^n_{>0} \rightarrow \R^n_{>0}$ and define the \emph{eigenspace} of $f$ to be  
For an order-preserving homogeneous function $f:\R^n_{>0} \rightarrow \R^n_{>0}$, we define the \emph{eigenspace} of $f$ to be
$$E(f) := \{x \in \R^n_{>0} : x \text{ is an eigenvector of } f\}.$$
Note that $E(f)$ only includes eigenvectors with all positive entries, if they exist. Determining whether or not $E(f)$ is nonempty can be a delicate problem. 

General results for establishing the existence and uniqueness of eigenvectors of such functions have been proved by Morishima \cite{Morishima64}, Oshime \cite{Oshime83}, and Nussbaum \cite{Nussbaum86,Nussbaum88,Nussbaum89}. These results require strong assumptions on $f$ in order to guarantee both existence and uniqueness of the positive eigenvectors of $f$. More general sufficient conditions for the existence of eigenvectors in $\R^n_{>0}$ have been established by Schneider and Turner \cite{ScTu72}, Nussbaum \cite{Nussbaum88,Nussbaum89} and Gaubert and Gunawardena \cite{GaGu04}.  
%All of these conditions prove the existence of a positive eigenvector by either directly or indirectly showing that $f$ has an invariant subset in $\R^n_{>0}$ that is bounded in Hilbert's projective metric. 
All of these conditions actually imply that the eigenspace $E(f)$ is both nonempty and bounded in Hilbert's projective metric. 
\emph{Hilbert's projective metric on} $\R^n_{>0}$ is defined by
$$d_H(x,y) :=  \log \max_{i, j \in [n]} \left( \frac{y_i \, x_j}{x_i \, y_j } \right),$$
where $[n] := \{1, \ldots, n\}$. 
We will describe Hilbert's projective metric in more detail in the next section, for now we just note that any order-preserving homogeneous map $f$ on $\R^n_{>0}$ is \emph{nonexpansive} with respect to $d_H$, that is $d_H(f(x),f(y)) \le d_H(x,y)$ for all $x, y \in \R^n_{>0}$.  

One of the main results of this paper is a computable necessary and sufficient condition for the eigenspace $E(f)$ to be bounded and nonempty in Hilbert's projective metric. This improves on a non-deterministic necessary and sufficient condition \cite{LLN16} and also generalizes a hypergraph condition \cite{AkGaHo20} that is sufficient for the existence of positive eigenvectors.  Below we give a brief description of these prior results.  %Two recent papers \cite{AkGaHo20,LLN16} have studied the eigenvector existence problem by explicitly seeking necessary and sufficient conditions for certain invariant subsets to be bounded in $(\R^n_{>0},d_H)$. We give a brief summary of some of the main results of these papers here.  

In \cite{GaGu04}, the authors describe three special classes of invariant subsets of $\R^n_{>0}$.  For any $\alpha, \beta > 0$, the \emph{sub-eigenspace} corresponding to $\alpha$ is the set
$$S_\alpha(f) := \{x \in \R^n_{>0} : \alpha x \le f(x)\}$$
and the \emph{super-eigenspace} corresponding to $\beta$ is
$$S^\beta(f) := \{x \in \R^n_{>0} : f(x) \le \beta x\}.$$
%A third type of invariant set for $f$ is formed by intersecting a super and a sub-eigenspace.  Such sets are called \emph{slice spaces} and are denoted
The intersection $S_\alpha^\beta(f) := S_\alpha(f) \cap S^\beta(f)$ is called a \emph{slice space}.
%$$S^\beta_\alpha(f) = \{x \in \R^n_{>0} : \alpha x \le f(x) \le \beta x\}.$$
Proving that any one of these sets is nonempty and bounded in Hilbert's projective metric is sufficient to guarantee that $f$ has a positive eigenvector. One of the main results in \cite{GaGu04} is a necessary and sufficient condition for all super-eigenspaces of $f$ to be $d_H$-bounded.  A similar condition is also be given for sub-eigenspaces.  Later, \cite{AkGaHo20} gave a necessary and sufficient condition for all slice spaces for $f$ to be $d_H$-bounded.  All three of these conditions can be described in terms of two directed hypergraphs $\Hm$ and $\Hp$ that were introduced in \cite{AkGaHo15}.

Recall that a \emph{directed hypergraph} is a set of nodes $N$ together with a collection of directed hyperarcs.  A \emph{directed hyperarc} is an ordered pair $(\mathbf{t},\mathbf{h})$ where $\mathbf{t}, \mathbf{h} \subseteq N$. We refer to $\mathbf{t}$ as the \emph{tail} of the hyperarc and $\mathbf{h}$ as the \emph{head}. The two hypergraphs $\Hm$ and $\Hp$ both have nodes $N = [n]$, and a set of directed hyperarcs $(I,\{j\})$ where $I \subset [n]$ and $j \in [n]$.  The hyperarcs of $\Hm$ are the pairs $(I,\{j\})$ such that $j \notin I$ and 
$$\lim_{t \rightarrow \infty} f(\exp(-te_I))_j = 0.$$  
Here $\exp$ is the entrywise natural exponential function and $e_I \in \R^n$ is the vector with entries 
$$(e_I)_i := \begin{cases} 
1 & \text{ if } i \in I \\
0 & \text{ if } i \in I^c.
\end{cases}$$
As we will see in the next section, if $f:\R^n_{>0} \rightarrow \R^n_{>0}$ is order-preserving and homogeneous, then $f$ extends continuously to $\R^n_{\ge 0}$, so the condition defining the hyperarcs of $\Hm$ can be replaced by $f(e_{[n] \bs I})_j = 0$. We will use $[n] \bs I$ and $I^c$ interchangeably for subsets of $[n]$. The hyperarcs of $\Hp$ are $(I, \{j\})$ such that $j \notin I$ and
$$\lim_{t \rightarrow \infty} f(\exp(te_I))_j = \infty.$$  
We say that a subset $I \subseteq [n]$ is \emph{invariant} in a hypergraph if there are no hyperarcs from a subset of $I$ to a subset of $I^c$. Since $f$ is order-preserving, if there is a hyperarc in $\Hm$ (or $\Hp)$ from a subset of $I$ to $j \in I^c$, then $(I,\{j\})$ is also a hyperarc of $\Hm$ (or $\Hp$).  Therefore $I$ is invariant in one of the graphs $\Hm$ or $\Hp$ if and only if there are no hyperarcs $(I,\{j\})$ in that hypergraph.

Although not originally described in terms of hypergraphs, \cite[Theorem 5]{GaGu04} says the following.
\begin{theorem}[Gaubert-Gunawardena] \label{thm:GG}
Let $f: \R^n_{>0} \rightarrow \R^n_{>0}$ be order-preserving and homogeneous.  Then all super-eigenspaces $S^\beta(f)$ are bounded in $(\R^n_{>0},d_H)$ if and only if $\Hp$ has no nonempty invariant sets $J \subsetneq [n]$.   
\end{theorem}

A corresponding condition involving the hypergraph $\Hm$ is equivalent to all sub-eigenspaces of $f$ being $d_H$-bounded.  An even more general result was proved as part of \cite[Theorem 1.2]{AkGaHo20}.
\begin{theorem}[Akian-Gaubert-Hochart] \label{thm:AGH}
Let $f: \R^n_{>0} \rightarrow \R^n_{>0}$ be order-preserving and homogeneous. 
%For all diagonal matrices $D \in M_n$ having positive diagonal entries, the map $D f$ has a nonempty and $d_H$-bounded set of eigenvalues 
All slice spaces $S_\alpha^\beta(f)$ are bounded in $(\R^n_{>0}, d_H)$ if and only if for every pair of nonempty disjoint sets $I, J \subset [n]$, either $I^c$ is not invariant in $\Hm$ or $J^c$ is not invariant in $\Hp$.
\end{theorem}

Both of these theorems are sufficient to prove that there exists an eigenvector of $f$ in $\R^n_{>0}$.  In fact, they both guarantee that the eigenspace $E(f)$
is nonempty and bounded in Hilbert's projective metric.  The condition of Theorem \ref{thm:GG} is more restrictive but its conclusion is stronger since if all super-eigenspaces are $d_H$-bounded, then so are all slice spaces. Even the conclusion that all slice spaces are bounded in $(\R^n_{>0},d_H)$ is quite strong.  In \cite{AkGaHo20}, it is shown to be equivalent to several other conditions on the function $f$, including that for any $n$-by-$n$ diagonal matrix $D$ with positive diagonal entries, the map $D f$ has a nonempty and $d_H$-bounded set of eigenvectors in $\R^n_{>0}$. 

Another recent result in nonlinear Perron-Frobenius theory is \cite[Theorem 5.1]{LLN16} which directly addresses when the eigenspace $E(f)$ is nonempty and $d_H$-bounded:  

\begin{theorem}[Lemmens-Lins-Nussbaum] \label{thm:LLN}
Let $f: \R^n_{>0} \rightarrow \R^n_{>0}$ be order-preserving and homogeneous. The eigenspace $E(f)$ is nonempty and bounded in $(\R^n_{>0},d_H)$ if and only if for every nonempty proper subset $J \subset [n]$, there exists $x \in \R^n$ such that 
\begin{equation} \label{illum}
\max_{j \in J} \frac{f(x)_j}{x_j} < \min_{i \in J^c} \frac{f(x)_i}{x_i}.
\end{equation}  
\end{theorem}

%\hl{Point out that one direction of the LLN theorem is proved using topological degree theory (the direction which says it is necessary to have an illuminating set in order to have nonempty and bounded $E(f)$).} 
Theorem \ref{thm:LLN} actually implies slightly more, since it shows that small perturbations of $f$ will also have a $d_H$-bounded set of eigenvectors, see \cite[Corollary 3.5]{LLN16}. %Note that the proof that \eqref{illum} is necessary for $E(f)$ to be nonnempty and bounded in Theorem \ref{thm:LLN} uses a topological degree theory argument.

In general the conditions of these three theorems are progressively more difficult to check.  In \cite{LLN16}, the authors suggest a non-deterministic algorithm to check the conditions of Theorem \ref{thm:LLN}.  The algorithm involves testing random points in $\R^n_{>0}$ until \eqref{illum} has been verified for every nonempty $J \subsetneq [n]$.  If the algorithm halts, then $f$ has a positive eigenvector. But this approach is less satisfactory than a straightforward combinatorial test, particularly when the function $f$ has symbolic parameters with unspecified values.   

In Section \ref{sec:exist} of this paper, we use Theorem \ref{thm:LLN} to prove a new necessary and sufficient condition for the eigenspace $E(f)$ of an order-preserving homogeneous function $f: \R^n_{>0} \rightarrow \R^n_{>0}$ to be nonempty and bounded in Hilbert's projective metric.  This new condition (Theorem \ref{thm:super}) is easier to check than the condition of Theorem \ref{thm:LLN}.  It also complements the hypergraph conditions of Theorems \ref{thm:GG} and \ref{thm:AGH}, allowing one to check the hypergraph conditions first, and then check the new condition only for subsets of $[n]$ where the hypergraph condition fails. Section \ref{sec:conv} shows how the results of Section \ref{sec:exist} can be improved for multiplicatively convex functions. We give simple necessary and sufficient conditions for the eigenspace $E(f)$ to be nonempty and bounded in $(\R^n_{>0}, d_H)$ when $f$ is multiplicatively convex. We also generalize a theorem of Hu and Qi \cite[Theorem 5]{HuQi16} which gives necessary and sufficient conditions for the existence of entrywise positive eigenvectors for strongly nonnegative tensors.

In Section \ref{sec:unique} we prove that if an order-preserving and homogeneous function $f: \R^n_{>0} \rightarrow \R^n_{>0}$ is also real analytic and the eigenspace $E(f)$ is nonempty and bounded in $(\R^n_{>0},d_H)$, then $f$ has a unique positive eigenvector up to scaling. This is precisely analogous to the situation for nonnegative matrices. By \emph{real analytic}, we mean that each entry of $f$ is a real analytic function on all of $\R^n_{>0}$.  To prove this result, we show that a real analytic nonexpansive map on a Banach space with the fixed point property cannot have a bounded set of fixed points with more than one element.
%We also prove the following general result for nonexpansive maps on a real Banach space. If $X$ is a real Banach space with the fixed point property and $f:X \rightarrow X$ is nonexpansive, real analytic, and has more than one fixed point, then the set of fixed points of $f$ is unbounded in $X$.
%Many examples of order-preserving homogeneous maps on $\R^n_{>0}$ in applications have real analytic entries, so we do not need any additional assumptions to prove uniqueness of positive eigenvectors once we know that $E(f)$ is $d_H$-bounded. 

When an order-preserving and homogeneous map $f: \R^n_{>0} \rightarrow \R^n_{>0}$ has a unique positive eigenvector $u \in \R^n_{>0}$ with $\|u\|=1$, it is sometimes important to know whether the normalized iterates $f^k(x)/\|f^k(x)\|$ converge to $u$ for all $x \in \R^n_{>0}$ as $k \rightarrow \infty$.  In Section \ref{sec:iterates} we give a new condition on the derivative of $f$ at $u$ that is sufficient to guarantee that the iterates converge.  We also point out that if $f$ has an eigenvector in $\R^n_{>0}$, then the normalized iterates of $f+\id$ always converge to an eigenvector of $f$, even if the normalized iterates of $f$ do not. 

We conclude in Section \ref{sec:examples} with detailed examples of how to use these new results in applications.  These include a population biology model and a Shapley operator associated with a stochastic game.  In Theorem \ref{thm:Mplus}, we give complete necessary and sufficient conditions for existence and uniqueness of entrywise positive eigenvectors for functions in the class $\mathcal{M}_+$ which includes the order-preserving homogeneous functions associated with the H-eigenvector problem for nonnegative tensors. We also translate the results of Section \ref{sec:exist} to the setting of topical functions, which are functions $T: \R^n \rightarrow \R^n$ that are order-preserving and additively homogeneous.

\section{Preliminaries}

%The set $\R^n_{\ge 0}$ is a closed, convex, pointed cone with nonempty interior in $\R^n$. By \emph{pointed cone}, we mean  

%For $x, y \in \R^n_{\ge 0}$, we say that $x$ is \emph{comparable} to $y$ if there are constants $\alpha, \beta > 0$ such that $\alpha x \le y \le \beta x$.  It is easy to see that comparability is an equivalence relation on $\R^n_{\ge 0}$ and therefore it partitions $\R^n_{\ge 0}$ into equivalence classes which we will call \emph{parts}.  One part of $\R^n_{\ge 0}$ is the interior $\R^n_{>0}$.  

A \emph{closed cone} in $\R^n$ is a closed convex set $K$ such that (i) $t K \subseteq K$ for all $t > 0$ and (ii) $K \cap (-K) = \{0\}$. A closed cone $K \subset \R^n$ induces a partial ordering $x \le_K y$ when $y-x \in K$. We write $x \ll_K y$ when $y-x$ is contained in the interior of $K$, denoted $\inter K$. The \emph{standard cone} in $\R^n$ is the closed cone $\R^n_{\ge 0}$ and we will write $x \le y$ for the partial ordering induced by the standard cone, and likewise, $x \ll y$ means that $x_i < y_i$ for all $i \in [n]$.

If $K$ is a closed cone and $x,y \in K$, then we say that $x$ is \emph{comparable} to $y$ and write $x \sim_K y$ if there exist constants $\alpha, \beta > 0$ such that $\alpha x \le_K y \le_K \beta x$.  It is easy to see that comparability is an equivalence relation on $K$. The equivalence classes under this relation are called the \emph{parts} of $K$. If $K$ has nonempty interior, then $\inter K$ is a part of $K$. For $x \in \R^n_{\ge 0}$, we define the \emph{support} of $x$ to be 
$$\supp(x) := \{j \in [n] : x_j > 0\}.$$  
The parts of $\R^n_{\ge 0}$ are the subsets $\R^J_{>0} = \{ x \in \R^n_{\ge 0} : \supp(x)=J \}$.    

Let $K \subset \R^n$ and suppose that $x, y \in K$ with $x \sim_K y$ and $y \ne 0$.  \emph{Hilbert's projective metric} is 
$$d_H(x,y) := \log \left( \frac{M(x/y)}{m(x/y)} \right) $$
where 
$$M(x/y) := \inf \{ \beta > 0 : x \le_K \beta y \}$$
and 
$$m(x/y) := \sup \{ \alpha > 0 : \alpha y \le_K x \}.$$
We adopt the convention that $d_H(0,0) = 0$ and $d_H(x,y) = \infty$ if $x$ and $y$ are not comparable. When $K = \R^n_{\ge 0}$ and $x, y\in \R^n_{>0}$, 
$$m(x/y) = \min_{i \in [n]} x_i/y_i,\hspace*{0.5cm}  M(x/y) = \max_{i \in [n]} x_i/y_i.$$
%and
Hilbert's projective metric has the following properties \cite[Proposition 2.1.1]{LemmensNussbaum}.  

\begin{proposition} \label{prop:dH}
Let $K$ be a closed cone in $\R^n$.  Then 
\begin{enumerate}
\item $d_H(x,y) \ge 0$ and $d_H(x,y) = d_H(y,x)$ for all $x, y \in K$.  
\item $d_H(x,z) \le d_H(x,y) + d_H(y,z)$ for any $x, y, z$ in a part of $K$.  
\item $d_H(\alpha x, \beta y) = d_H(x,y)$ for all $\alpha, \beta >0$ and $x, y \in K$.  
\end{enumerate} 
\end{proposition}

The use of Hilbert's projective metric to prove results in linear Perron-Frobenius theory can be traced back independently to Birkhoff \cite{Birkhoff57} and Samelson \cite{Samelson57}.  Their methods also apply to nonlinear Perron-Frobenius theory because if $f: K \rightarrow K$ is order-preserving with respect to $\le_K$ and homogeneous, then $f$ is nonexpansive with respect to $d_H$ \cite[Proposition 2.1.3]{LemmensNussbaum}.

Let $K$ be a closed cone in $\R^n$.  
Let $f:K \rightarrow K$ be an order-preserving homogeneous function. The \emph{cone spectral radius} of $f$ is 
$$r_K(f) := \lim_{k \rightarrow \infty} \|f^k\|_K^{1/k}$$
where $\|f\|_K := \sup_{x \in K} \|f(x)\|/\|x\|$.  If $K$ has nonempty interior, then an equivalent formula for the cone spectral radius is 
\begin{equation} \label{rlim}
r_K(f) = \lim_{k \rightarrow \infty} \|f^k(x)\|^{1/k}
\end{equation}
which does not depend on the choice of $x \in \inter K$ \cite[Proposition 5.3.6]{LemmensNussbaum}.  Note that neither formula for $r_K(f)$ depends on which norm is used for $\R^n$.  

%The following well known result is due to Kre\u{\i}n and Rutman \cite[Theorem 9.1]{KrRu48}.  Their original result is actually more general since it applies to compact functions on cones in a Banach space and has a somewhat more general homogeneity assumption as well. 
%\begin{theorem}[Kre\u{\i}n-Rutman] \label{thm:KR}
Although it can be difficult to know whether or not an order-preserving homogeneous function has an eigenvector in the interior of $K$, the following well-known result guarantees the existence of an eigenvector in $K$ \cite[Corollary 5.4.2]{LemmensNussbaum}.  
\begin{theorem} \label{thm:KR}
Let $K$ be a closed cone in $\R^n$ and let $f:K \rightarrow K$ be order-preserving and homogeneous.  Then there exists a nonzero $x \in K$ such that $f(x) = r_K(f) x$.  
\end{theorem}

An order-preserving homogeneous function $f: K \rightarrow K$ can have several eigenvectors and more than one eigenvalue. However, if $f$ has two eigenvectors in the same part of $K$, then those eigenvectors must have the same eigenvalue \cite[Corollary 5.2.2]{LemmensNussbaum}. Furthermore, $r_K(f)$ is the maximum of all of the eigenvalues of $f$ \cite[Proposition 5.3.6]{LemmensNussbaum}. Although Theorem \ref{thm:KR} does not guarantee that $f$ has an eigenvector in the interior of $K$, if $x \in \inter K$ is an eigenvector of $f$,  then the eigenvalue corresponding to $x$ must be $r_K(f)$.  

Suppose $K \subset \R^n$ is a closed cone with nonempty interior and $f: K \rightarrow K$ is order-preserving and homogeneous. The \emph{upper Collatz-Wielandt number} for $f$ is  
\begin{equation} \label{cw}
r(f) := \inf_{x \in \inter K} M(f(x)/x),
\end{equation}
and the \emph{lower Collatz-Wielandt number} for $f$ is
\begin{equation} \label{lcw}
\lambda(f) := \sup_{x \in \inter K} m(f(x)/x).
\end{equation} 

From the definition, it is clear that $\lambda(f) \le r(f)$ always. It is well known that the upper Collatz-Wielandt number $r(f)$ is always equal to the \emph{cone spectral radius} $r_K(f)$ \cite[Theorem 5.6.1]{LemmensNussbaum}.

If an order-preserving and homogeneous function $f:\R^n_{\ge 0} \rightarrow \R^n_{\ge 0}$ has a positive eigenvector $x \in \R^n_{>0}$ such that $f(x) = \mu x$, then the eigenvalue $\mu$ must equal both the upper and lower Collatz-Wielandt numbers, i.e., $r(f) = \lambda(f)=\mu$. In fact, we will prove a stronger result in Lemma \ref{lem:mindisp}. Another quick observation about Collatz-Wielandt numbers is the following.

\begin{lemma} \label{lem:cworder}
Let $f, g: \R^n_{\ge 0} \rightarrow \R^n_{\ge 0}$ be order-preserving and homogeneous. If $f(x) \le g(x)$ for all $x \in \R^n_{\ge 0}$, then $r(f) \le r(g)$ and $\lambda(f) \le \lambda(g)$.
\end{lemma}

\begin{proof}
This follows immediately from the definitions \eqref{cw} and \eqref{lcw}. 
\end{proof}

Every order-preserving homogeneous function $f: \R^n_{> 0} \rightarrow \R^n_{> 0}$ extends continuously to the closed cone $\R^n_{\ge 0}$. In fact, the following result \cite[Corollary 2]{BuSp00} (see also \cite[Corollary 4.6]{BuSpNu03} and \cite[Theorem 5.1.5]{LemmensNussbaum}) actually says a little more. Note that when working with the extended real line $\overline{\R} = [-\infty, \infty]$, we use the usual order topology induced by the natural linear ordering on $\overline{\R}$. 

\begin{theorem}[Burbanks-Sparrow] \label{thm:BS}
Let $f: \R^n_{>0} \rightarrow \R^n_{>0}$ be order-preserving and homogeneous. Then $f$ extends continuously to a map $\underline{f}:\R^n_{\ge 0} \rightarrow \R^n_{\ge 0}$.  The map $f$ also extends continuously to a map $\overline{f}:(0,\infty]^n \rightarrow (0,\infty]^n$. Both extensions are order-preserving and homogeneous.
\end{theorem}

We will write $f$ instead of $\overline{f}$ or $\underline{f}$ when working with these extensions, as long as the meaning is clear. Note that when we say that $\overline{f}$ is order-preserving and homogeneous, we are using the usual conventions that $c < \infty$ for all $c \in \R$ and $c \cdot \infty = \infty$ for all $c > 0$.   
We can make the extension of $f$ to $(0,\infty]^n$ somewhat more familiar by observing that $\overline{f} = L \underline{g} L$ where $g = L f L$ and $L:[0,\infty]^n\rightarrow [0,\infty]^n$ is the entrywise reciprocal function 
\begin{equation} \label{L}
L(x)_j := \begin{cases}
x_j^{-1} & \text{ if } 0 < x_j < \infty \\
\infty & \text{ if } x_j = 0 \\
0 & \text{ if } x_j = \infty.
\end{cases}
\end{equation}
For any order-preserving homogeneous function $f:(0,\infty]^n \rightarrow (0,\infty]^n$, we can define the upper and lower Collatz-Wielandt numbers $r(f)$ and $\lambda(f)$ using \eqref{cw} and \eqref{lcw}, and we allow the possibility that they might be infinite. We define the \emph{support} of $x \in (0,\infty]^n$ to be $\supp(x) = \{j \in [n] : x_j < \infty \}$. We can also define the \emph{parts} of $(0,\infty]^n$ to be the sets 
$$(0,\infty]^J := \{x \in (0,\infty]^n : \supp(x) = J \}.$$ 
Note that Hilbert's projective metric is defined on each part and is finite for any pair of vectors in the same part. 

Using the reciprocal function $L$, we see that $\lambda(f) = r(LfL)^{-1}$ and $r(f) = \lambda(LfL)^{-1}$. Because $f$ extends continuously to $(0,\infty]^n$, the hypergraph $\Hp$ 
has a hyperarc $(I,\{j\})$ with $j \notin I$ if and only if $f(\omega_I)_j = \infty$ where $\omega_I \in (0,\infty]^n$ is given by
$$(\omega_I)_i := \begin{cases}
\infty & \text{ if } i \in I \\
1 & \text{ if } i \in I^c.
\end{cases}$$
Since $\omega_I = L(e_{[n] \bs I})$, it follows that $\Hp = \mathcal{H}^-_0(LfL)$. 

The following iterative formula for the lower Collatz-Wielandt number is similar to \eqref{rlim}.

\begin{lemma} \label{lem:llim}
Let $f: \R^n_{\ge 0} \rightarrow \R^n_{\ge 0}$ (or $f: (0,\infty]^n \rightarrow (0,\infty]^n$) be order-preserving and homogeneous. For every $x \in \R^n_{>0}$,
$$\lambda(f) = \lim_{k \rightarrow \infty} \left(\min_{i \in [n]} f^k(x)_i\right)^{1/k}.$$
\end{lemma}

\begin{proof}
If $f(x)_i = 0$ for some $i \in [n]$ and $x \in \R^n_{>0}$, then $\lambda(f) = 0$ by definition.  At the same time, $f^k(x)_i = 0$ for all $k \in \N$ since $f$ is order-preserving and there exists some $\beta > 0$ such that $f(x) \le \beta x$.  Also, $f^k(y)_i = 0$ for all $y \in \R^n_{>0}$ since $x$ and $y$ are comparable. Therefore $\lim_{k \rightarrow \infty} \left( \min_{i \in [n]} f^k(y)_i \right)^{1/k} = 0$ for all $y \in \R^n_{>0}$.  

Suppose now that $f(x)_i > 0$ for all $i \in [n]$. Let $g = L f L$. Observe that
\begin{align*}
r(g) = \inf_{x \in \R^n_{>0}} \max_{i \in [n]} \frac{g(x)_i}{x_i} &= \inf_{y = Lx \in \R^n_{>0}} \max_{i \in [n]}  \frac{y_i}{f(y)_i}  \\
&= \left(\sup_{y = Lx \in \R^n_{>0}} \min_{i \in [n]} \frac{f(y)_i}{y_i}\right)^{-1} = \lambda(f)^{-1}.
\end{align*}
For $x \in \R^n_{>0}$ and $k \in \N$, $\min_{i \in [n]} f^k(x)_i = (\max_{i \in [n]} g^k(y)_i)^{-1}$ , where $y = Lx$ and $k \in \N$. Working with the supremum norm $\|x\|_\infty = \max_{i \in [n]} |x_i|$, we have by \eqref{rlim}:
$$\lim_{k \rightarrow \infty} \left(\min_{i \in [n]} f^k(x)_i \right)^{1/k}= \lim_{k \rightarrow \infty} \left(\max_{i \in [n]} g^k(y)_i \right)^{-1/k}  = r(g)^{-1} = \lambda(f).$$
\end{proof} 

For an order-preserving homogeneous function $f:\R^n_{>0} \rightarrow \R^n_{>0}$, we define the \emph{minimum displacement} of $f$ to be $\delta(f) := \inf \{ d_H(x,f(x)) : x \in \R^n_{>0} \}$.  

\begin{lemma} \label{lem:mindisp}
Let $f:\R^n_{>0} \rightarrow \R^n_{>0}$ be order-preserving and homogeneous.  Then the minimum displacement of $f$ satisfies 
$$\delta(f) = \log r(f) - \log \lambda(f).$$ 
\end{lemma}
\begin{proof}
By \eqref{rlim} and Lemma \ref{lem:llim}, the following is true for any $x \in \R^n_{>0}$.
\begin{align*}
\log r(f) - \log \lambda(f) &= \lim_{k \rightarrow \infty} \frac{1}{k} \log \left(\max_{i \in [n]} f^k(x)_i \right) - \lim_{k \rightarrow \infty} \frac{1}{k} \log \left(\min_{j \in [n]} f^k(x)_j\right) \\
&= \lim_{k \rightarrow \infty} \frac{1}{k} \log \left(\max_{i, j \in [n]} \frac{f^k(x)_i}{f^k(x)_j} \right) \\
&= \lim_{k \rightarrow \infty} \frac{1}{k} \log \left(\max_{i, j \in [n]} \frac{f^k(x)_i \, x_j}{f^k(x)_j \, x_i} \right) \\
&= \lim_{k \rightarrow \infty} \frac{1}{k} d_H(x,f^k(x)).
\end{align*}
It is known that $\lim_{k \rightarrow \infty} \frac{1}{k} d_H(x,f^k(x)) = \delta(f)$ for all $x \in \R^n_{>0}$ \cite[Theorem 1]{GaVi12}, so we conclude that $\log r(f) - \log \lambda(f) = \delta(f)$.  
\end{proof}

If $f$ has an eigenvector $x \in \R^n_{>0}$, then $d_H(x,f(x)) = 0$, so $\delta(f) = 0$. In that case Lemma \ref{lem:mindisp} proves that $r(f) = \lambda(f)$.  However, there are simple examples where $\delta(f) = 0$, but $f$ has no eigenvectors in $\R^n_{>0}$.

Another interesting corollary of Theorems \ref{thm:KR} and \ref{thm:BS} is the following. A similar slightly weaker observation was made in \cite[Theorem 2]{ChFrLi13}.  

\begin{corollary} \label{cor:formaleig}
Let $f: \R^n_{>0} \rightarrow \R^n_{>0}$ be order-preserving and homogeneous.  Then $f$ has an eigenvector $y \in (0,\infty]^n$ with $y_i < \infty$ for some $i \in [n]$ such that $f(y) = \lambda(f) y$. 
\end{corollary}
\begin{proof}
Note that $LfL$ is order-preserving and homogeneous on $\R^n_{>0}$ and it has a continuous extension to $\R^n_{\ge 0}$ by Theorem \ref{thm:BS}. Furthermore, the continuous extension is order-preserving and homogeneous and $r(LfL) = \lambda(f)^{-1}$. Then, by Theorem \ref{thm:KR}, there is a nonzero $x \in \R^n_{\ge 0}$ such that $LfL(x) = \lambda(f)^{-1}x$.  So $y = Lx$ is an eigenvector of the continuous extension of $f$ to $(0,\infty]^n$ with eigenvalue $\lambda(f)$, and since $x \ne 0$, there is at least one entry of $y$ that is not infinite.
\end{proof}

\section{Existence of eigenvectors} \label{sec:exist}

\subsection{Boundary Collatz-Wielandt numbers and eigenvectors}
Before stating the main result of this section, we need to introduce a class of auxiliary functions defined for any order-preserving homogeneous $f: \R^n_{>0} \rightarrow \R^n_{>0}$.  For $\alpha \in [-\infty, \infty]$ and $J \subseteq [n]$, let $P^J_\alpha: [-\infty,\infty]^n \rightarrow [-\infty,\infty]^n$ be the projection 
\begin{equation} \label{PJ}
P^J_\alpha(x)_j := \begin{cases}
x_j & \text{ if } j \in J \\
\alpha & \text{ otherwise.}
\end{cases}
\end{equation}
Observe that for all $x \in [0,\infty]^n$, $P^J_0 x \le x \le P^J_\infty x$. For any order-preserving homogeneous function $f:\R^n_{>0} \rightarrow \R^n_{>0}$, we define 
$$f^J_0 := P^J_0 f P^J_0 \hspace*{0.5cm} \text{ and } \hspace*{0.5cm} f^J_{\infty} := P^J_\infty f P^J_\infty.$$
By Theorem \ref{thm:BS}, both $f^J_0: \R^n_{\ge 0} \rightarrow \R^n_{\ge 0}$ and $f^J_\infty: (0,\infty]^n \rightarrow (0,\infty]^n$ are well-defined order-preserving and homogeneous functions.  Note also that 
\begin{equation} \label{monotonicity}
f^J_0(x) \le f(x) \le f^J_\infty(x)
\end{equation}
for all $x \in \R^n_{>0}$.   
Now the main result of this section is:

\begin{theorem} \label{thm:super}
Let $f: \R^n_{>0} \rightarrow \R^n_{>0}$ be order-preserving and homogeneous.  The eigenspace $E(f)$ is nonempty and bounded in $(\R^n_{>0}, d_H)$ if and only if for all nonempty proper subsets $J \subset [n]$, 
\begin{equation} \label{superIllum}
r(f^J_0) < \lambda(f^{[n] \bs J}_\infty).
\end{equation}
\end{theorem}

Before proving Theorem \ref{thm:super} we need the following lemmas.  

\begin{lemma} \label{lem:near}
Let $f:\R^n_{>0} \rightarrow \R^n_{>0}$ be order-preserving and homogeneous.  For any $v \in \R^J_{>0}$ and $w \in \R^{[n] \bs J}_{>0}$, let $x = v+tw$ where $t > 0$ is a small constant. Then 
$$\lim_{t \rightarrow 0} \frac{f(x)_j}{x_j} = \frac{f^J_0(v)_j}{v_j} \text{ and } \lim_{t \rightarrow 0} \frac{f(x)_i}{x_i} =\frac{f^{[n] \bs J}_\infty(w)_i}{w_i}$$
for all $j \in J$ and $i \in J^c$. 
\end{lemma}

\begin{proof}
The first limit follows immediately from Theorem \ref{thm:BS}.  The second follows from Theorem \ref{thm:BS} with the additional observation that
$$\lim_{t \rightarrow 0} \frac{f(x)_i}{x_i} = \lim_{t \rightarrow 0} \frac{f(t^{-1} x)_i}{t^{-1} x_i} = \frac{f^{[n] \bs J}_\infty(w)_i}{w_i}$$
for all $i \in J^c$. 
\end{proof}

\begin{lemma} \label{lem:super}
Let $f: \R^n_{>0} \rightarrow \R^n_{>0}$ be order-preserving and homogeneous. Let $J$ be a nonempty proper subset of $[n]$.  Then there exists $x \in \R^n_{>0}$ such that 
$$\max_{j \in J} \frac{f(x)_j}{x_j} < \min_{i \in J^c} \frac{f(x)_i}{x_i}$$
if and only if $r(f^J_0) < \lambda(f^{[n] \bs J}_\infty).$
\end{lemma} 
\begin{proof}
($\Rightarrow$) Suppose that $\max_{j \in J} {f(x)_j}/{x_j} < \min_{i \in J^c} {f(x)_i}/{x_i}$ for some $x \in \R^n_{>0}$. By the definitions of the upper and lower Collatz-Wielandt numbers and \eqref{monotonicity}, 
\begin{align*}r(f^J_0) \le \max_{j \in J} \frac{f^J_0(x)_j}{x_j} &\le \max_{j \in J} \frac{f(x)_j}{x_j} \\
&< \min_{i \in J^c} \frac{f(x)_i}{x_i} \le \min_{i \in J^c} \frac{f^{[n] \bs J}_\infty(x)_i}{x_i} \le  \lambda(f^{[n] \bs J}_\infty).
\end{align*}

\noindent
($\Leftarrow$) Suppose that $r(f^J_0) < \lambda(f^{[n] \bs J}_\infty)$. 
By the definitions of the upper and lower Collatz-Wielandt numbers, we may choose $v \in \R^J_{>0}$ and $w \in \R^{[n] \bs J}_{>0}$ such that $\max_{j \in J} f^J_0(v)_j/v_j$ is close enough to $r(f^J_0)$ and $\min_{i \in J^c} f^{[n] \bs J}_\infty(w)_i/w_i$ is close enough to $\lambda(f^{[n] \bs J}_\infty)$ so that
$$\max_{j \in J} \frac{f^J_0(v)_j}{v_j}< \min_{i \in J^c} \frac{f^{[n] \bs J}_\infty(w)_i}{w_i}.$$
Then, by Lemma \ref{lem:near}, there exists $x \in \R^n_{>0}$ such that 
$$\max_{j \in J} \frac{f(x)_j}{x_j} < \min_{i \in J^c} \frac{f(x)_i}{x_i}.$$
\end{proof}

\begin{proof}[Proof of Theorem \ref{thm:super}]
Theorem \ref{thm:super} follows immediately from Theorem \ref{thm:LLN} and Lemma \ref{lem:super}.  
\end{proof}

The following result shows that it is not always necessary to check 
the conditions of Theorem \ref{thm:super}
for every $J \subset [n]$. 

\begin{theorem} \label{thm:quick}
Let $f: \R^n_{>0} \rightarrow \R^n_{>0}$ be order-preserving and homogeneous. Suppose that \eqref{superIllum} holds for some proper nonempty subset $J \subset [n]$.  If $f^J_0$ has a nonempty and $d_H$-bounded set of eigenvectors in $\R^J_{>0}$, then \eqref{superIllum} also holds for all nonempty $I \subseteq J$.  If $f^{[n] \bs J}_\infty$ has a nonempty and $d_H$-bounded set of eigenvectors in $\R^{[n] \bs J}_{>0}$, then \eqref{superIllum} holds for all $I \supseteq J$ with $I \ne [n]$.
\end{theorem}

\begin{proof}
Suppose that $f^J_0$ has a nonempty and $d_H$-bounded set of eigenvectors in $\R^J_{>0}$.  Consider any nonempty proper subset $I \subset J$. Let $D: \R^n \rightarrow \R^n$ be the linear transformation given by 
$$D(x)_j := \begin{cases}
x_j & \text{ if } j \in I \\
(1-\epsilon)x_j & \text{ if } j \in J \bs I \\
0 & \text{ otherwise}
\end{cases}$$
where $\epsilon>0$ is a small constant.
Note that $D f^J_0 (x) \le f^J_0 (x)$ for all $x \in \R^n_{\ge 0}$, so $r(D f^J_0) \le r(f^J_0)$ by Lemma \ref{lem:cworder}.
As long as $\epsilon$ is sufficiently small, Theorem \ref{thm:LLN} will imply that $D f^J_0$ has a nonempty and $d_H$-bounded set of eigenvectors in $\R^J_{>0}$ since all of the inequalities \eqref{illum} which hold for the map $f^J_0$ will also hold for $Df^J_0$. Let $v \in \R^J_{>0}$ be one such eigenvector. Then $D f^J_0(v) = r(D f^J_0) v$. This means that 
$$\frac{f^J_0(v)_j}{v_j} = \begin{cases}
r(Df^J_0) & \text{ if } j \in I \\
(1-\epsilon)^{-1} r(D f^J_0) & \text{ if } j \in J \bs I.\\
\end{cases}$$ 
If $\epsilon > 0$ is small enough, then by \eqref{superIllum},
$$(1-\epsilon)^{-1} r(Df^J_0) < \lambda(f^{[n] \bs J}_\infty).$$
By the definition of the lower Collatz-Wielandt number, we can choose $w \in \R^{[n] \bs J}_{>0}$ such that 
$$(1-\epsilon)^{-1} r(Df^J_0) < \min_{j \in J^c} \frac{f^{[n] \bs J}_\infty(w)_j}{w_j}.$$
Then
$$\max_{i \in I} \frac{f^J_0(v)_i}{v_i} = r(Df^J_0) < (1-\epsilon)^{-1} r(Df^J_0) = \min_{i \in J\bs I} \frac{f^J(v)_i}{v_i} < \min_{j \in J^c} \frac{f^{[n] \bs J}_\infty(w)_j}{w_j}.$$
By Lemma \ref{lem:near}, we can choose $x \in \R^n_{>0}$ such that
$$\max_{i \in I}\frac{f(x)_i}{x_i} < \min_{j \in I^c} \frac{f(x)_j}{x_j}.$$
From this and \eqref{monotonicity}, we conclude that
\begin{align*}r(f^I_0) \le \max_{j \in I} \frac{f^J_0(x)_j}{x_j} &\le \max_{j \in I} \frac{f(x)_j}{x_j} \\
&< \min_{i \in I^c} \frac{f(x)_i}{x_i} \le \min_{i \in I^c} \frac{f^{[n] \bs I}_\infty(x)_i}{x_i} \le  \lambda(f^{[n] \bs I}_\infty).
\end{align*}
The proof that \eqref{superIllum} holds for all $I \supseteq J$ when $f^{[n] \bs J}_\infty$ has a nonempty and $d_H$-bounded set of eigenvectors in $\R^{[n] \bs J}_{>0}$ is essentially the same.  
\end{proof}

In order to check the conditions of Theorem \ref{thm:super}, the following lemma can also be helpful.  

\begin{lemma} \label{lem:AsubB}
Let $f: \R^n_{>0} \rightarrow \R^n_{>0}$ be order-preserving and homogeneous. Suppose that $A \subseteq B \subseteq [n]$. Then 
\begin{enumerate}[(a)]
\item $r(f^A_0) \le r(f^B_0)$, \text{ and }
\item $\lambda(f^A_\infty) \ge \lambda(f^B_\infty)$.  
\end{enumerate} 
\end{lemma}

\begin{proof}
Note that $f^A_0(x) \le f^B_0(x)$ for all $x \in \R^n_{\ge 0}$.  Therefore $r(f^A_0) \le r(f^B_0)$ by Lemma \ref{lem:cworder}. 
%The lower Collatz-Wielandt numbers for $f^A_\infty$ and $f^B_\infty$ are defined using \eqref{lcw} so Lemma \ref{lem:cworder} also applies these functions. 
Since $f^A_\infty(x) \ge f^B_\infty(x)$ for all $x \in (0, \infty]^n$, a similar argument shows that $\lambda(f^A_\infty) \ge \lambda(f^B_\infty)$.  
%Alternatively use:
%$$\lambda(f^A_\infty) = r(L f^A_\infty L)^{-1} \ge r(L f^B_\infty L)^{-1} = \lambda(f^B_\infty).$$
%For all $i \in A$ and $x \in \R^B_{>0}$, $f_0^A(P^A x)_i \le f_0^B(x)_i$ by \eqref{monotonicity}. Therefore,
%\begin{align*}
%r(f_0^A) = \inf_{y \in \R^A_{>0}} \max_{i \in A} \frac{f_0^A(y)_i}{y_i} &= \inf_{x \in \R^B_{>0}} \max_{i \in A} \frac{f_0^A(P^A x)_i}{x_i} \\
%&\le \inf_{x \in \R^B_{>0}} \max_{i \in B} \frac{f_0^B(x)_i}{x_i}  = r(f_0^B).
%\end{align*} 
%Similarly, for all $i \in A$ and $x \in \R^B_{>0}$, $f_\infty^A(P^A x)_i \ge f_\infty^B(x)_i$. Therefore,
%\begin{align*}
%\lambda(f_\infty^A) = \sup_{y \in \R^A_{>0}} \min_{i \in A} \frac{f_\infty^A(y)_i}{y_i} &= \sup_{x \in \R^B_{>0}} \min_{i \in A} \frac{f_\infty^A(P^A x)_i}{x_i} \\
%&\ge \sup_{x \in \R^B_{>0}} \min_{i \in B} \frac{f_\infty^B(x)_i}{x_i}  = \lambda(f_\infty^B).
%\end{align*} 
\end{proof}

The following condition can show that an order-preserving homogeneous function has no eigenvectors in $\R^n_{>0}$.  

\begin{theorem} \label{thm:none}
Let $f: \R^n_{>0} \rightarrow \R^n_{>0}$ be order-preserving and homogeneous.  If 
$$r(f^J_0) > \lambda (f^{[n] \bs J}_\infty)$$
for some nonempty proper subset $J \subset [n]$, then $f$ has no eigenvectors in $\R^n_{>0}$. 
\end{theorem}

\begin{proof}
By Lemma \ref{lem:AsubB}, we have
$$r(f) \ge  r(f^J_0) > \lambda (f^{[n] \bs J}_\infty) \ge \lambda(f).$$
Then the minimum displacement $\delta(f) = r(f) - \lambda(f) > 0$ by Lemma \ref{lem:mindisp}. This means that $f$ cannot have an eigenvector in $\R^n_{>0}$.  
\end{proof}

Notice that there is a middle ground between the conditions of Theorems \ref{thm:super} and \ref{thm:none} where $r(f^J_0) \le \lambda(f^{[n] \bs J}_\infty)$ for all nonempty proper $J \subset [n]$, but $r(f^J_0) = \lambda(f^{[n] \bs J}_\infty)$ for at least one $J$.  In that case, it is possible that $f$ has no eigenvectors in $\R^n_{>0}$, or that the eigenspace $E(f)$ is unbounded in $(\R^n_{>0}, d_H)$. For example consider the linear maps corresponding to the nonnegative matrices 
$$\begin{bmatrix} 1 & 0 \\ 0 & 1 \end{bmatrix} \text{ and } \begin{bmatrix} 1 & 1 \\ 0 & 1 \end{bmatrix}.$$
The first is the identity matrix and every element of $\R^2_{>0}$ is an eigenvector, while the second matrix has no eigenvectors in $\R^2_{>0}$.  The interested reader can easily check that $r(f^{\{1\}}_0) = \lambda(f^{\{2\}}_\infty)=1$ and $r(f^{\{2\}}_0) \le \lambda(f^{\{1\}}_\infty)$ for each matrix.

\begin{remark} \label{rem:irred}
If the function $f$ in Theorem \ref{thm:super} is differentiable and the Jacobian matrix $f'(x)$ is irreducible for all $x \in \R^n_{>0}$, then any eigenvector  of $f$ in $\R^n_{>0}$ will be unique (up to scaling). This is also true, even if $f'(x)$ is not irreducible, as long as $f'(x)$ has a unique positive eigenvector for every $x \in \R^n_{>0}$ \cite[Corollary 6.4.8]{LemmensNussbaum}.  In that case, the condition of Theorem \ref{thm:super} is both necessary and sufficient for $f$ to have an eigenvector in $\R^n_{>0}$. 
\end{remark}

%%% NEXT PART IS DEFINITELY TRUE, BUT PROBABLY NOT WORTH INCLUDING.
%%%\color{blue}
%%%The following is an alternative formulation of Theorem \ref{thm:super}.  
%%%\begin{theorem} \label{thm:altmain}
%%%Let $f: \R^n_{>0} \rightarrow \R^n_{>0}$ be order-preserving and homogeneous.  The eigenspace $E(f)$ is nonempty and bounded in $(\R^n_{>0}, d_H)$ if and only if for all nonempty proper subsets $J \subset N$, either 
%%%\begin{equation} \label{altsuper}
%%%r(f^J_0) < \lambda(f) ~~\text{ or }~~ r(f) < \lambda(f^{N \bs J}_\infty).
%%%\end{equation}
%%%\end{theorem}
%%%
%%%\begin{proof}
%%%By Lemma \ref{lem:AsubB}, $r(f^J_0) \le r(f)$ and $\lambda(f) \le \lambda(f^{N \bs J}_\infty)$. So either inequality in \eqref{altsuper} would imply that $r(f^J_0) < \lambda(f^{N \bs S}_\infty)$. Therefore, if \eqref{altsuper} is true for all nonempty $J \subsetneq N$, then $E(f)$ is nonempty and bounded in $(\R^n_{>0}, d_H)$. Conversely, if $E(f)$ is nonempty and bounded in $(\R^n_{>0}, d_H)$, then $r(f^J_0) \le r(f) = \lambda(f) \le \lambda(f^{N \bs J}_\infty)$ and $r(f^J_0) < \lambda(f^{N \bs J}_\infty)$ for all nonempty $J \subsetneq N$. Together, these imply \eqref{altsuper} for all $J$.    
%%%\end{proof} 
%%%\color{black}

\subsection{Using hypergraphs to detect eigenvectors} \label{sec:hypergraph}

The results in the previous subsection can be combined with the hypergraph conditions from Theorems \ref{thm:GG} and \ref{thm:AGH} to give a very general method for checking whether an order-preserving homogeneous $f: \R^n_{>0} \rightarrow \R^n_{>0}$ has any eigenvectors in $\R^n_{>0}$.  The following two lemmas show how the upper and lower Collatz-Wielandt numbers of the auxiliary functions $f^J_0$ and $f^{[n] \bs J}_\infty$ relate to the hypergraphs $\Hm$ and $\Hp$.  For a hypergraph $\mathcal{H}$ and a subset $J$ of the nodes $N$, the \emph{reach} of $J$ in $\mathcal{H}$, denoted $\reach(J,\mathcal{H})$, is the smallest invariant subset of $N$ that contains $J$. Recall from Section \ref{sec:intro} that a set of nodes $I \subseteq [n]$ is invariant in one of the hypergraphs $\Hm$ or $\Hp$ if and only if there are no hyperarcs $(I, \{j\})$ leaving $I$. 

\begin{lemma} \label{lem:hyperorder}
Let $f: \R^n_{>0} \rightarrow \R^n_{>0}$ be order-preserving and homogeneous and let $J \subset [n]$.  If $\reach(J,\Hp) = I$, then $\lambda(f^{[n] \bs J}_\infty) = \lambda(f^{[n] \bs I}_\infty)$.  If $\reach(J^c, \Hm) = I^c$, then $r(f^J_0) = r(f^I_0)$.
\end{lemma}

\begin{proof}
Fix any $x \in \R^n_{>0}$. Let $x^k = (f^J_0)^k(x)$ for $k \in \N$. Since there exists $\beta > 0$ such that $f^J_0(x) \le \beta x$ and $f$ is order-preserving, it follows that $x^{k+1} \le \beta x^k$ for all $k \in \N$. In particular, $\supp(x^{k+1}) \subseteq \supp(x^k)$ for each $k$. Let $I_k = \supp(x^k)$.  Note that $I_{k+1} = I_k$ if and only if $I^c_k$ is invariant in $\Hm$. So the sets $I_k$ eventually stabilize at some $k = m$ where $I_m = \reach(J^c,\Hm)^c$. Let $I = I_m$. We can choose $y \in \R^n_{>0}$ such that $P^I_0y = x^m$, and then $x^k = (f^I_0)^{k-m}(y)$ for all $k \ge m$.  Therefore $r(f^J_0) = r(f^I_0)$ by \eqref{rlim}.  

The proof that $\lambda(f^{[n] \bs J}_\infty) = \lambda(f^{[n] \bs I}_\infty)$ when $I = \reach(J,\Hp)$ is essentially the same, but uses Lemma \ref{lem:llim} in place of \eqref{rlim}.  
\end{proof}

%Lemma \ref{lem:hyperorder} lets us make the following observation.

\begin{lemma} \label{lem:connect}
Let $f: \R^n_{>0} \rightarrow \R^n_{>0}$ be order-preserving and homogeneous. Consider any $J \subset [n]$.  The following are equivalent. 
\begin{enumerate}[(a)]
\item \label{item:a} There is no nonempty $A \subseteq J$ such that $A^c$ is invariant in $\Hm$.
\item \label{item:b} $\reach(J^c,\Hm) = [n]$.
\item \label{item:c} $r(f^J_0) = 0$.
\end{enumerate}
Similarly, the following are also equivalent.
\begin{enumerate}[(a),resume]
\item \label{item:d} There is no nonempty $B \subseteq J^c$ such that $B^c$ is invariant in $\Hp$.
\item \label{item:e} $\reach(J,\Hp) = [n]$.
\item \label{item:f} $\lambda(f^{[n] \bs J}_\infty) = \infty$.
\end{enumerate}
\end{lemma}

\begin{proof}
It is obvious from the definition of $\reach(J^c,\Hm)$ that conditions \ref{item:a} and \ref{item:b} are equivalent. 

\ref{item:b} $\Rightarrow$ \ref{item:c}.  If $\reach(J^c,\Hm) = [n]$, then $r(f^J_0) = r(f^\varnothing_0) = 0$.  

\ref{item:c} $\Rightarrow$ \ref{item:a}. Suppose that $r(f^J_0) = 0$. If $A \subseteq J$ and $f^J_0(e_A)_j > 0$ for all $j \in A$, then there exists $\alpha > 0$ such that $f^J_0(e_A) \ge \alpha e_A$.  But then $(f^J_0)^k(e_A) \ge \alpha^k e_A$ and so $r(f^J_0) \ge \alpha$ by \eqref{rlim}. This is a contradiction, so we conclude that $f^J_0(e_A)_j = 0$ for some $j \in A$. Therefore there is a hyperarc from $A^c$ into $A$ in $\Hm$, so $A^c$ is not invariant in $\Hm$.   

To prove that conditions \ref{item:d}, \ref{item:e}, and \ref{item:f} are equivalent, let $g = L f L$, where $L$ is the entrywise reciprocal function from \eqref{L}. Observe that $\Hp = \mathcal{H}^-_0(g)$ and $f^{[n] \bs J}_\infty = L g^{[n] \bs J}_0 L$, so $\lambda(f^{[n] \bs J}_\infty) = \infty$ if and only if $r(g^{[n] \bs J}_0) = 0$. Then \ref{item:d}, \ref{item:e}, \ref{item:f} are equivalent to each other because \ref{item:a}, \ref{item:b}, \ref{item:c} are equivalent.
\end{proof}

The following corollary is a restatement of Theorem \ref{thm:AGH} using Lemma \ref{lem:connect}.

\begin{corollary} \label{cor:AGH2}
Let $f: \R^n_{>0} \rightarrow \R^n_{>0}$ be order-preserving and homogeneous. Then the following are equivalent. 
\begin{enumerate}
\item All slice spaces $S_\alpha^\beta(f)$ are $d_H$-bounded.
\item For every $J \subset [n]$, either $r(f^J_0) = 0$ or $\lambda(f^{[n] \bs J}_\infty) = \infty$. 
\item For every $J \subset [n]$, either $\reach(J^c,\Hm) = [n]$ or $\reach(J,\Hp) = [n]$.
\end{enumerate}
\end{corollary}

Corollary \ref{cor:AGH2} suggests the following strategy for checking whether the eigenspace $E(f)$ is nonempty and bounded in $(\R^n_{>0},d_H)$.  First, check whether $\Hp$ has any invariant subsets $J$.  If not, then $E(f)$ is nonempty and $d_H$-bounded by Theorem \ref{thm:GG}.  If there are invariant sets $J$ in $\Hp$, then for each one, check whether $\reach(J^c,\Hm)=[n]$.  If there are any sets $J \subset [n]$ such that $\reach(J, \Hp) \ne [n]$ and $\reach(J^c, \Hm) \ne [n]$, then those are the only sets where we need to check \eqref{superIllum} in Theorem \ref{thm:super}. We give examples in Section \ref{sec:examples} to demonstrate this process and show how we can further reduce the number of sets $J$ to check by using Theorem \ref{thm:quick} and Lemma \ref{lem:AsubB}. 

For some order-preserving homogeneous functions $f:\R^n_{>0} \rightarrow \R^n_{>0}$ there are faster ways to confirm that $f$ has a positive eigenvector.  The \emph{directed graph associated with} $f$ is the digraph $\G(f)$ with vertices $[n]$ and an arc from $i$ to $j$ when 
$$\lim_{t \rightarrow \infty} f(\exp(te_{\{j\}}))_i = \infty.$$
Note that $(i,j)$ with $i \ne j$ is a arc of $\G(f)$  if and only if $(\{j\},\{i\})$ is a hyperarc of $\Hp$. Also observe that the direction of the hyperarc $(\{j\},\{i\}) \in \Hp$ is the reverse of $(i,j) \in \G(f)$.  
For a nonnegative matrix $A= \begin{bmatrix} a_{ij} \end{bmatrix}_{i, j \in [n]}$, $\G(A)$ is equivalent to the usual directed graph associated with a matrix which has an arc from $i$ to $j$ precisely when $a_{ij} \ne 0$. 

In \cite[Theorem 2]{GaGu04}, the authors point out that if $\G(f)$ is strongly connected, then $f$ has an eigenvector in $\R^n_{>0}$. This directly generalizes the classical Perron-Frobenius theorem for irreducible matrices, since $A \in \R^{n \times n}_{\ge 0}$ is irreducible if and only if $\G(A)$ is strongly connected. This nonlinear version of the Perron-Frobenius theorem follows immediately from Theorem \ref{thm:GG}. If $\G(f)$ is strongly connected, then $\Hp$ has no invariant subsets, so Theorem \ref{thm:GG} implies all super-eigenspaces of $f$ are bounded in $(\R^n_{>0}, d_H)$. As described in the introduction, this guarantees that $E(f)$ is nonempty and bounded in $(\R^n_{>0},d_H)$.  If $\G(f)$ is not strongly connected, then its nodes can be partitioned into strongly connected components.  A strongly connected component is called a \emph{final class} if no arcs leave the component.  The following sufficient condition for existence of a positive eigenvector generalizes \cite[Theorem 2]{GaGu04} and is generally easier to check than Theorem \ref{thm:super} since the graph $\G(f)$ and its connected components can be computed relatively quickly even when the dimension $n$ is large.  
%We will prove the following generalization of \cite[Theorem 2]{GaGu04}.  

\begin{theorem} \label{thm:graphCond}
Let $f: \R^n_{>0} \rightarrow \R^n_{>0}$ be order-preserving and homogeneous. If $\G(f)$ has a unique final class $C$ and $r(f^{[n] \bs C}_0) < r(f)$, then the eigenspace $E(f)$ is nonempty and bounded in $(\R^n_{>0}, d_H)$.    
\end{theorem}   

\begin{proof}
Suppose temporarily that $\lambda(f) < r(f)$. We may assume without loss of generality that $r(f) = 1$ by replacing $f$ with $r(f)^{-1}f$. By Theorem \ref{thm:KR}, there is an eigenvector $v \in \R^n_{\ge 0}$ such that $f(v) = v$.  By Corollary \ref{cor:formaleig}, there is also an eigenvector $w \in (0,\infty]^n$ such that $f(w) = \lambda(f) w$.  Let $J = \supp(v)$ and $I = \supp(w)$. For any $x \in \R^n_{>0}$ we can find $\alpha, \beta > 0$ such that $\alpha v \le x \le \beta w$. Then $\alpha v \le f^k(x) \le \beta \lambda(f)^k w$. This means that $\lim_{k \rightarrow \infty} f^k(x)_i = 0$ for all $i \in I$, while $f^k(x)_j$ is bounded below for all $j \in J$.  Therefore $I$ and $J$ are disjoint. Note that $r(f^J_0) = 1$. Since $r(f^{[n] \bs C}_0) < r(f) = 1$, Lemma \ref{lem:AsubB} implies that $J$ cannot be a subset of $[n] \bs C$. So there exists some $j \in J \cap C$. Since $C$ is the unique final class of $\G(f)$, there must be a path in $\G(f)$ from every $i \in [n]$ to $j$. Choose $i \in I$ and let $m$ be the length of a path from $i$ to $j$ in $\G(f)$. Then by the definition of $\G(f)$, we must have
$$\lim_{t \rightarrow \infty} f^m(\exp(te_{\{j\}}))_i = \infty.$$
At the same time, since $j \notin I$, there is a $\beta > 0$ such that $\exp(te_{\{j\}}) \le \beta w$ for all $t>0$.  Then 
$$f^m(\exp(te_{\{j\}}))_i \le \beta \lambda(f)^m w_i < \beta w_i < \infty$$ 
for all $t > 0$.  This is a contradiction, so we conclude that $\lambda(f) = r(f) = 1$.  

Let $A$ be a nonempty proper subset of $[n]$. If $A \subseteq [n] \bs C$, then 
$$r(f^A_0) \le r(f^{[n] \bs C}_0) < r(f) = \lambda(f) \le \lambda(f^{[n] \bs A}_\infty)$$
by Lemma \ref{lem:AsubB}. Therefore \eqref{superIllum} holds for all $A \subseteq [n] \bs C$.  On the other hand, if $A \cap C \ne \varnothing$, then there is a $j \in A \cap C$ and a path in $\Hp$ from $j$ to every $i \in [n]$.  This means that $\reach(A,\Hp) = [n]$, so $\lambda(f^{[n] \bs A}_\infty) = \infty$ by Lemma \ref{lem:connect}. Then \eqref{superIllum} also holds for $A$.  Since \eqref{superIllum} holds for all nonempty proper $A \subset [n]$, we conclude by Theorem \ref{thm:super} that $E(f)$ is nonempty and bounded in $(\R^n_{>0},d_H)$.
\end{proof}

\begin{remark}
The number of conditions to check in Theorem \ref{thm:super} grows exponentially as the dimension $n$ increases. %Sometimes the number of computations needed can be reduced as in Theorems \ref{thm:quick} and \ref{thm:graphCond}. However, 
If $\text{P} \ne \text{NP}$, then there cannot be a polynomial time necessary and sufficient condition for $E(f)$ to be nonempty and bounded in $(\R^n_{>0},d_H)$ that applies to all order-preserving homogeneous maps $f: \R^n_{>0} \rightarrow \R^n_{>0}$.  This follows from a result of Yang and Zhao \cite{YaZh04}. They study a special class of order-preserving homogeneous min-max operators called monotone boolean functions. Every monotone boolean function has a trivial eigenvector, however \cite[Proposition 1]{YaZh04} proves that determining whether the trivial eigenvector is the only eigenvector (up to scaling) is co-NP hard. For monotone boolean functions, the eigenspace is $d_H$-bounded if and only if the trivial eigenvector is the only eigenvector up to scaling. 
\end{remark}

Although determining whether general order-preserving homogeneous functions have a nonempty and $d_H$-bounded eigenspace quickly becomes computationally intractable when the dimension gets large, we will see in the next section that the situation improves considerably when our functions have an additional convexity property. 

\section{Multiplicatively convex functions.} \label{sec:conv}

Many order-preserving homogeneous functions $f:\R^n_{>0} \rightarrow \R^n_{>0}$ have the additional property that $\log \circ f \circ \exp$ is convex. Recall that a function $g:D \rightarrow \R$ is \emph{convex} when its \emph{epigraph} 
$$\epi(f) := \{ (x,y) \in D \times \R : y \ge f(x) \}$$ 
is convex. A vector-valued function $f:D \rightarrow \R^n$ is \emph{convex} if each entry of $f$ is convex. When $\log \circ f \circ \exp$ is convex, we will say that $f$ is \emph{multiplicatively convex}.  Note that $f:\R^m_{>0} \rightarrow \R^n_{>0}$ is multiplicatively convex if and only if for all $x, y \in \R^m_{>0}$ and $0 \le \theta \le 1$, 
$$f(x^\theta y^{1-\theta}) \le f(x)^{\theta} f(y)^{1-\theta} $$
where this notation is understood to indicate entrywise products and powers. We will continue to use this notation throughout the rest of this paper.  

The following result is a special case of \cite[Proposition 6.1]{Nussbaum86}. The proof is short, so we include it here.
\begin{lemma} \label{lem:prop61}
Let $f, g:\R^m_{>0} \rightarrow \R^n_{>0}$ be multiplicatively convex.  Then $f+g$ is multiplicatively convex, and if $f$ is also order-preserving and $m = n$, then $f \circ g$ is multiplicatively convex.  
\end{lemma}

\begin{proof}
By H{\"o}lder's inequality,
\begin{align*}
f(x^\theta y^{1-\theta})_i + g(x^\theta y^{1-\theta})_i &\le f(x)_i^\theta f(y)_i^{1-\theta} + g(x)_i^\theta g(y)_i^{1-\theta} \\ 
 &\le [f(x)_i + g(x)_i]^\theta [f(y)_i + g(y)_i]^{1-\theta}
\end{align*}
for each $i \in [n]$ and $0 < \theta < 1$. Therefore $f+g$ is multiplicatively convex. If $f$ is also order-preserving and $m = n$, then 
\begin{align*}
f(g(x^\theta y^{1-\theta})) &\le f(g(x)^\theta g(y)^{1-\theta})\\
&\le f(g(x))^{\theta} f(g(y))^{1-\theta}, 
\end{align*}
which proves that $f \circ g$ is multiplicatively convex.
\end{proof}

An immediate consequence of Lemma \ref{lem:prop61} is that any order-preserving linear map $f:\R^m_{>0} \rightarrow \R^n_{>0}$ is multiplicatively convex.  
Other important classes of order-preserving, homogeneous, multiplicatively convex functions include matrices over the max-times algebra (see e.g., \cite{MuPe15}), and the class $\mathcal{M}_+$ introduced in \cite{Nussbaum89} (see also \cite[Proposition 3.1 and Equation 3.15]{Nussbaum86} and \cite[Section 6.6]{LemmensNussbaum}). We will discuss the class $\mathcal{M}_+$ in more detail in Section \ref{sec:examples}.

When $f:\R^n_{>0} \rightarrow \R^n_{>0}$ is order-preserving, homogeneous, and multiplicatively convex, it is possible to give very general conditions for existence of eigenvectors in $\R^n_{>0}$, even when the eigenspace $E(f)$ might not be bounded in Hilbert's projective metric.  

Let $C_1, \ldots, C_m$ denote the strongly connected components of $\G(f)$.  Borrowing some terminology from the theory of nonnegative matrices (see e.g., \cite[Definition 2.3.8]{BermanPlemmons}), we will say that $C_j$ is a \emph{basic class} of $f$ if $r(f^{C_j}_0) = r(f)$.

In \cite{HuQi16}, Hu and Qi define what they call strongly nonnegative tensors. These tensors correspond to a special class of order-preserving, homogeneous, and multiplicatively convex functions on $\R^n_{>0}$. The following definition extends their notion of strong nonnegativity to all order-preserving, homogeneous, and multiplicatively convex functions on $\R^n_{>0}$. 

\begin{definition}
Let $f: \R^n_{>0} \rightarrow \R^n_{>0}$ be order-preserving, homogeneous, and multiplicatively convex. Let $\mathcal{C}$ denote the set of strongly connected components of $\G(f)$.  We will say that $f$ is \emph{strongly nonnegative} if $r(f^{C}_0) = r(f)$ for every final class $C \in \mathcal{C}$ and $r(f^{C}_0) < r(f)$ when $C \in \mathcal{C}$ is not a final class.  That is, $f$ is strongly nonnegative if and only if its basic and final classes coincide.
\end{definition}

%The main results of this section are the following two theorems. 
The following theorems are the main results of this section. The first theorem shows the relationship between the upper and lower Collatz-Wielandt numbers and the strongly connected components of $\G(f)$. 

\begin{theorem} \label{thm:basicClass}
Let $f: \R^n_{>0} \rightarrow \R^n_{>0}$ be order-preserving, homogeneous, and multiplicatively convex.  Let $\mathcal{C}$ denote the set of strongly connected components of $\G(f)$ and let $\mathcal{F}$ denote the set of final classes of $\G(f)$.  Then 
$$r(f) = \max_{C \in \mathcal{C}} r(f^{C}_0) ~~~~ \text{ and } ~~~~ \lambda(f) = \min_{C \in \mathcal{F}} r(f^{C}_0).$$
\end{theorem}

The second theorem gives very general conditions for the existence of eigenvectors when $f$ is multiplicatively convex, even when the eigenspace $E(f)$ might not be bounded in Hilbert's projective metric.  Note that conditions \ref{item:strongNonneg} and \ref{item:analyticConverse} generalize \cite[Theorem 5]{HuQi16} and corresponding results for reducible nonnegative matrices \cite[Theorem 2.3.10]{BermanPlemmons}.  

\begin{theorem} \label{thm:convMain}
Let $f:\R^n_{>0} \rightarrow \R^n_{>0}$ be order-preserving, homogeneous, and multiplicatively convex.  
\begin{enumerate}[(a)]
\item \label{item:convex} $E(f)$ is nonempty and bounded in $(\R^n_{>0},d_H)$ if and only if $f$ is strongly nonnegative and $\G(f)$ has only one final class.
\item \label{item:strongNonneg} If $f$ is strongly nonnegative, then $E(f)$ is nonempty.
\item \label{item:analyticConverse} If $f$ is real analytic and $E(f)$ is nonempty, then $f$ is strongly nonnegative.
\end{enumerate}
\end{theorem}

More is known for max-times functions: \cite[Theorem 3.4]{BaStVa95} and \cite[Theoreme 2.2.4]{Gaubert92} independently proved that a matrix in the max-times algebra has an entrywise positive eigenvector if and only if all of its final classes are basic. So it is possible for a multiplicatively convex function that is not strongly nonnegative to have an entrywise positive eigenvector if it is not real analytic. 

In general, it is much easier to check the conditions in Theorem \ref{thm:convMain} than the conditions of Theorem \ref{thm:super}. The strongly connected components of $\G(f)$ can be computed efficiently, and there are at most $n$ Collatz-Wielandt numbers to compute. For more details about the complexity aspects of this problem, see \cite[Section 4.2.3]{AkGaHo20}. See also \cite[Corollary 4.4]{AkGaHo20} which gives a sufficient condition for all slice spaces to be $d_H$-bounded when $f$ is multiplicatively convex.

To prove Theorems \ref{thm:basicClass} and \ref{thm:convMain}, we need some preliminary results. 

\begin{lemma} \label{lem:convConst}
Let $g: (0,\infty) \rightarrow (0,\infty)$ be order-preserving and multiplicatively convex.  Then $\lim_{x \rightarrow \infty} g(x) < \infty$ if and only if $g$ is constant. If $g$ is also real analytic, then $\lim_{x \rightarrow \infty} g(x) < \infty$ if and only if $g'(x) = 0$ for some $x > 0$.  
\end{lemma}

\begin{proof}
If $g$ is constant, then $\lim_{x \rightarrow \infty} g(x) < \infty$. Conversely, if $g$ is not constant, there exists $x < y$ such that $g(x) < g(y)$.  Then the epigraph of $\log \circ g \circ \exp$ has a support line with a positive slope.  It follows that $\lim_{x \rightarrow \infty} g(x) = \infty$. 
If $g$ is also real analytic and $g'(x) = 0$ for some $x \in \R$, then $g'(y) = 0$ for all $y < x$, by convexity.  But then $g'$ is identically zero on all of $(0,\infty)$ which means that $g$ is constant.
\end{proof}

\begin{lemma} \label{lem:convEq}
Let $f: \R^n_{>0} \rightarrow \R^n_{>0}$ be order-preserving, homogeneous, and multiplicatively convex.  Suppose that there are no arcs leaving $J$ in $\G(f)$ for some nonempty $J \subset [n]$. Then
$$f(x)_j = f^J_0(x)_j = f^J_\infty(x)_j$$
for all $x \in \R^n_{>0}$ and $j \in J$.  
\end{lemma}

\begin{proof}
No arcs leaving $J$ in $\G(f)$ means that $\lim_{t \rightarrow \infty} f(\exp(te_{\{i\}}))_j < \infty$ for all $j \in J$ and $i \in [n] \bs J$. Since $f$ is order-preserving and homogeneous, 
$$\lim_{t \rightarrow \infty} f(x \exp(te_{\{i\}}))_j \le (\max_{k \in [n]} x_k) \lim_{t \rightarrow \infty}  f(\exp(t e_{\{i\}}))_j < \infty$$ 
for all $x \in \R^n$. By Lemma \ref{lem:convConst} this means that $t \mapsto f(x\exp(t e_{\{i \}}))_j$ is constant.  Therefore the value of $f(x)_j$ does not change if we change the entries $x_i$ with $i \in [n] \bs J$.  
By letting $x_i \rightarrow 0$ (respectively, $x_i \rightarrow \infty$) for all $i \in [n] \bs J$, we get $f(x)_j = f^J_0(x)_j$ (and $f(x)_j = f^J_\infty(x)_j$) for all $j \in J$. 
\end{proof}

\begin{lemma} \label{lem:bdryEig}
Let $f: \R^n_{>0} \rightarrow \R^n_{>0}$ be order-preserving, homogeneous, and multiplicatively convex. If $C$ is a final class of $\G(f)$, then $f^C_0$ and $f^C_\infty$ both have eigenvectors with support equal to $C$ and eigenvalue equal to $r(f^C_0) = \lambda(f^C_\infty)$. 
\end{lemma}
\begin{proof}
The directed graph $\G(f^C_\infty)$ associated with $f^C_\infty$ on $\R^C_{>0}$ has nodes $C$ and contains all of the arcs between pairs $i, j \in C$ that are in $\G(f)$.  Since $C$ is a strongly connected component of $\G(f)$, we see that $\G(f^C_\infty)$ is strongly connected. Therefore $f^C_\infty$ has an eigenvector $v \in (0,\infty]^n$ with $\supp(v) = C$ and eigenvalue $\lambda(f^C_\infty)$ by Theorem \ref{thm:graphCond}. By Lemma \ref{lem:convEq}, $f^C_0(x)_j = f^C_\infty(x)_j$ for all $x \in \R^n_{>0}$ and $j \in J$. If we choose $x \in \R^n_{>0}$ such that $v = P^C_\infty x$, then $P^C_0 x$ is an eigenvector of $f^C_0$ with the same eigenvalue as $v$. This means that $r(f^C_0) =\lambda(f^C_\infty)$.  
\end{proof}

\begin{proof}[Proof of Theorem \ref{thm:basicClass}]
Let $C_1, \ldots, C_m$ denote the strongly connected components of $\G(f)$.
These components are the vertices of a directed acyclic graph with an arc from $C_i$ to $C_j$ whenever there is an arc in $\G(f)$ from a vertex in $C_i$ to a vertex in $C_j$.  Since the vertices of a directed acyclic graph have a topological ordering, we can assume that the strongly connected components of $\G(f)$ are ordered so that there are no paths from $C_i$ to $C_j$ in $\G(f)$ when $i < j$. We may also assume that the final classes are $C_1, \ldots, C_k$ where $k$ is the number of final classes in $\G(f)$.  

By Theorem \ref{thm:KR}, $f$ has an eigenvector $v \in \R^n_{\ge 0}$ with eigenvalue $r(f)$.  Let $j$ be the minimal index such that $C_j \cap \supp(v) \ne \varnothing$. Let $J = C_1 \cup \ldots \cup C_j$.   
By Lemma \ref{lem:convEq}, $f^J_0(v)_i = f(v)_i$ for all $i \in J$. Note that this equation is true even if some entries of $v$ are zero because both $f$ and $f^J_0$ extend continuously to the boundary of $\R^n_{\ge 0}$ by Theorem \ref{thm:BS}. %Extending the result of Lemma \ref{lem:convEq} to the boundary of $\R^n_{\ge 0}$ using Theorem \ref{thm:BS}, we have $f^J_0(v)_i = f(v)_i$ for all $i \in J$.
Therefore $P^J_0 v$ is an eigenvector of $f^J_0$ with eigenvalue $r(f)$.  Since $\supp(v) \subseteq C_j \cup \ldots \cup C_m$, $P^J_0v = P^{C_j}_0 v$ so $P^{C_j}_0 v$ is an eigenvector of $f^{C_j}_0$ with eigenvalue $r(f)$.  This implies that $r(f^{C_j}_0) \ge r(f)$.  Since $r(f^{C_i}_0) \le r(f)$ by Lemma \ref{lem:AsubB} for each $1 \le i \le m$, we see that $r(f) = \max_{1 \le i \le m} r(f^{C_i}_0)$. 

To prove the corresponding equation for $\lambda(f)$, note that $f$ also has an eigenvector $w \in (0,\infty]^n$ with eigenvalue $\lambda(f)$ by Corollary \ref{cor:formaleig}.  If there is a path from $j$ to $i$ in $\G(f)$, then $\lim_{t \rightarrow \infty} f^p(\exp(te_{\{i\}}))_j = \infty$ where $p$ is the length of the path.  If $w_i = \infty$, then there is a constant $\beta > 0$ such that $\exp(te_{\{i\}}) \le \beta w$ for all $t > 0$.  Therefore 
$$\beta \lambda(f)^p w_j = \beta f^p(w)_j \ge \lim_{t \rightarrow \infty} f^p(\exp(t e_{\{i\}}))_j = \infty $$ 
so $w_j = \infty$.  This means that if $j \in \supp(w)$ (so that $w_j < \infty$) and $i \notin \supp(w)$, then there cannot be a path from $j$ to $i$ in $\G(f)$.  It also means that if $j \in \supp(w)$, then any other vertex in the same strongly connected component with $j$ must also be in the support of $w$.  

Let $j$ be the smallest index such that $C_j \cap \supp(w) \ne \varnothing$.  By the above argument, there cannot be any arcs in $\G(f)$ from a vertex in $C_j$ to a vertex in $C_i$ when $i < j$.  Therefore $C_j$ is a final class of $\G(f)$ and $C_j \subseteq \supp(w)$.  

By Lemma \ref{lem:convEq}, $f_\infty^{C_j}(P^{C_j}_\infty w)_i = f(w)_i = \lambda(f) w_i$ for all $i \in C_j$.  Therefore $P^{C_j}_\infty w$ is an eigenvector of $f_\infty^{C_j}$ with eigenvalue $\lambda(f)$.  By Lemma \ref{lem:bdryEig}, $r(f^{C_j}_0) = \lambda(f^{C_j}_\infty) = \lambda(f)$.  Since $r(f^{C_i}_0) = \lambda(f^{C_i}_\infty) \ge \lambda(f)$ for all other final classes $C_i$ by Lemmas \ref{lem:AsubB} and \ref{lem:bdryEig}, it follows that $\lambda(f) = \min_{1 \le i \le k} r(f^{C_i}_0)$.  
\end{proof}

\begin{lemma} \label{lem:classCap}
Let $f: \R^n_{>0} \rightarrow \R^n_{>0}$ be order-preserving, homogeneous, and multiplicatively convex. 
%Let $J$ be the union of the final classes of $\G(f)$. Then $f$ is strongly nonnegative if and only if $r(f^{[n] \bs J}_0) < r(f)$ and $\lambda(f^J_{\infty}) = r(f)$. 
If $\G(f)$ has a unique final class $C$, then $f$ is strongly nonnegative if and only if $r(f^{[n] \bs C}_0) < r(f^C_0)$. 
\end{lemma}

\begin{proof}
Suppose that $r(f^{[n] \bs C}) < r(f^C_0)$. If $B$ is any non-final strongly connected component of $\G(f)$, then $B \subseteq [n] \bs C$.  Therefore $r(f^B_0) \le r(f^{[n] \bs C}_0) < r(f^C_0)$ by Lemma \ref{lem:AsubB} and $r(f^C_0) = r(f)$ by Theorem \ref{thm:basicClass}. Thus $f$ is strongly nonnegative.

Conversely, if $f$ is strongly nonnegative, then $r(f^B_0) < r(f) = r(f^C_0)$ for every strongly connected component $B$ of $\G(f)$ other than $C$.  The directed graph 
$\G(f^{[n] \bs C}_0)$ has nodes $[n] \bs C$ and an arc from $i$ to $j$ precisely when 
$$\lim_{t \rightarrow \infty} f^{[n] \bs C}_0(\exp(te_{\{j\}}))_i = \infty.$$
If there is an arc from $i$ to $j$ in $\G(f^{[n] \bs C}_0)$, then there is an arc from $i$ to $j$ in $\G(f)$ by \eqref{monotonicity}.  So for any $i \in [n] \bs C$, the strongly connected component $A[i]$ of $\G(f^{[n] \bs C}_0)$ which contains $i$ is a subset of the corresponding strongly connected component $B[i]$ of $\G(f)$ that contains $i$.  Thus $r(f^{A[i]}_0) \le r(f^{B[i]}_0)$ for all $i \in [n] \bs C$ by Lemma \ref{lem:AsubB}. Then, by Theorem \ref{thm:basicClass}, 
$$r(f^{[n] \bs C}_0) = \max_{i \in [n] \bs C} r(f^{A[i]}_0) \le \max_{i \in [n] \bs C} r(f^{B[i]}_0) < r(f^C_0).$$
\end{proof}

\begin{lemma} \label{lem:samederiv}
Let $f: \R^n_{>0} \rightarrow \R^n_{>0}$ be order-preserving, homogeneous, real analytic, and multiplicatively convex. Then $\G(f) = \G(f'(x))$ for every $x \in \R^n_{>0}$.   
\end{lemma}

\begin{proof}
By the chain rule, 
$$\left. \frac{\partial}{\partial t} f(x \exp( t e_{\{j\}}))_i \right|_{t = 0} = x_j f'(x)_{ij}$$
for any $x \in \R^n_{>0}$. 
Therefore, $\lim_{t \rightarrow \infty} f(x \exp(te_{\{j\}})_i = \infty$ if and only if $f'(x)_{ij} > 0$ by Lemma \ref{lem:convConst}.  Since $f$ is $d_H$-nonexpansive, $\lim_{t \rightarrow \infty} f(x \exp(te_{\{j\}}))_i = \infty$ if and only if $\lim_{t \rightarrow \infty} f(\exp(te_{\{j\}}))_i = \infty$ which is equivalent to $(i,j) \in \G(f)$.  Therefore $\G(f) = \G(f'(x))$.
\end{proof}

\begin{proof}[Proof of Theorem \ref{thm:convMain} \ref{item:convex}] Suppose that $f$ is strongly nonnegative and $\G(f)$ has a unique final class $C$. By Lemma \ref{lem:classCap}, $r(f^{[n] \bs C}_0) < r(f^C_0) = r(f)$.  Then $E(f)$ is nonempty and bounded in $(\R^n_{>0},d_H)$ by Theorem \ref{thm:graphCond}.  

Conversely, suppose that $E(f)$ is nonempty and $d_H$-bounded. Let $C$ be a final class of $\G(f)$. Then $r(f^C_0) = \lambda(f^C_\infty)$ by Lemma \ref{lem:bdryEig}. Since $E(f)$ is nonempty, we also know that $\lambda(f) = r(f)$ by Lemma \ref{lem:mindisp}. Using Lemma \ref{lem:AsubB}, we have 
 $$\lambda(f) \le  \lambda(f^C_\infty) = r(f^C_0) \le r(f)$$
which implies that $r(f^C_0) = r(f)$ so $C$ is a basic class.  If $\G(f)$ had another final class $B \subset [n]$ different from $C$, then the same argument would show that $r(f^B_0) = r(f) = r(f^C_0)$.  But $B \subseteq [n] \bs C$, so 
$$r(f^B_0) \le r(f^{[n] \bs C}_0) < \lambda(f^C_\infty) = r(f^C_0)$$
by Theorem \ref{thm:super} and Lemma \ref{lem:AsubB}.  So we conclude that $C$ must be the only final class of $\G(f)$. Since $C$ is the unique final class and $r(f^{[n] \bs C}_0) < r(f^C_0)$, Lemma \ref{lem:classCap} implies that $f$ is strongly nonnegative.    
\end{proof}

\begin{proof}[Proof of Theorem \ref{thm:convMain} \ref{item:strongNonneg}] 
Let $J$ be the union of the final classes of $\G(f)$. Since $f$ is strongly nonnegative, each final class $C$ is basic.  By Lemma \ref{lem:bdryEig}, there is an eigenvector in $\R^C_{>0}$ with eigenvalue $r(f)$. Let $u \in \R^n_{\ge 0}$ be the sum of one such eigenvector from each final class.  Observe that $\supp(u) = J$ and $u$ is an eigenvector of $f$ with eigenvalue $r(f)$ by Lemma \ref{lem:convEq}.   

We can assume without loss of generality that $[n] \bs J = \{1, \ldots, m\}$ and $J = \{m+1, \ldots, n\}$ for some $m < n$.  Let $K = \{x+tu : x \in \R^{[n] \bs J}_{\ge 0} \text{ and } t\ge 0\}$. Then $K$ is a closed cone contained in $\R^n_{\ge 0}$ and $f$ leaves $K$ invariant by Lemma \ref{lem:convEq}.  We can identify $K$ with $\R^{m+1}_{\ge 0}$ via the invertible map $V: \R^{m+1}_{\ge 0} \rightarrow K$ defined by 
$$V(x)_i := \begin{cases}
x_i & \text{ when } i \le m \\
x_{m+1} u_i & \text{ when } i > m. 
\end{cases}$$
Note that $V$ is linear and order-preserving. Let $g = V^{-1} \circ f \circ V$.  Then $g:\R^{m+1}_{>0} \rightarrow \R^{m+1}_{>0}$ is order-preserving, homogeneous, and multiplicatively convex.  The digraph $\G(g)$ has a unique final class $B = \{m+1\}$. The other strongly connected components of $\G(g)$ have vertices in $[n] \bs J$ and are exactly the same as the corresponding strongly connected components of $\G(f)$. Since $r(g^B_0) = r(f^J_0) = r(f)$ and $r(g^A_0) = r(f^A_0) < r(f)$ for any other strongly connected component $A$ of $\G(g)$, it follows that $g$ is strongly nonnegative.  Therefore, Theorem \ref{thm:convMain} \ref{item:convex} implies that $g$ has an eigenvector $w \in \R^{m+1}_{>0}$. Then $V(w) \in \R^{n}_{>0}$ is an eigenvector of $f$. 
\end{proof} 

\begin{proof}[Proof of Theorem \ref{thm:convMain} \ref{item:analyticConverse}] Suppose that $f$ is real analytic and $u \in \R^n_{>0}$ is an eigenvector of $f$. We may also assume that $r(f) = 1$ by scaling $f$.  Thus, $f(u) = u$. Since $f$ is order-preserving, $A := f'(u)$ is a nonnegative matrix; and since $f$ is homogeneous, $u$ is an eigenvector of $A$ with eigenvalue $r(f) = 1$.  By Lemma \ref{lem:samederiv}, $\G(A) = \G(f)$.  So $\G(f)$ and $\G(A)$ have the same strongly connected components and the same final classes. 

Let $J$ be the union of the final classes of $\G(f)$ and let $I = [n] \bs J$. We can assume without loss of generality that $I = \{1, \ldots, m\}$ and $J = \{m+1, \ldots, n\}$ for some $m < n$. Then we can write $A$ and $u$ in block form: 
$$A = \left[ \begin{array}{c|c}
A_{I} & * \\ \hline
0 & A_{J}
\end{array} \right] ~~~~ \text{ and } ~~~~ u = \left[ \begin{array}{c}
u_{I} \\ 
u_{J}
\end{array} \right].$$

Observe that $A_J u_J = u_J$. Let $x = \begin{bmatrix} 0 \\ u_J \end{bmatrix}$. Then $Ax \ge x$.
For every $i \in I$, there is a path in $\G(A)$ from $i$ to some $j \in J$. Therefore, there is some $p \in \N$ such that $A^p x \gg 0$.  Then, since $x_j > 0$ for all $j \in J$, it follows that $A^p e_J \gg 0$ as well.   This means that $A^p (u-\epsilon e_J) \ll u$ for all $\epsilon > 0$.  Note that $A^p$ is the derivative of $f^p$ at $u$ by the chain rule for Fr\'{e}chet derivatives.  So we have $f^p(u-\epsilon e_J) \ll u$ when $\epsilon > 0$ is sufficiently small.  Then 
$$(f^{[n] \bs J}_0)^p(u) = (f^{[n] \bs J}_0)^p( u - \epsilon e_J ) \le f^p( u - \epsilon e_J ) \ll u.$$
Therefore there is a positive constant $\beta < 1$ such that $(f^{[n] \bs J}_0)^p(u) \le \beta u$.  By \eqref{rlim}, this means that $r(f^{[n] \bs J}_0) < 1$.  

Let $C$ be a strongly connected component of $\G(f)$. If $C$ is not final, then $C \subseteq [n] \bs J$, so Lemma \ref{lem:AsubB} shows that $r(f^{C}_0) \le r(f^{[n] \bs J}_0) < 1 = r(f)$.  If $C$ is final, then $P^{C}_0 u$ is an eigenvector of $f^{C}_0$ with eigenvalue 1 by Lemma \ref{lem:convEq}.  Therefore the final classes are exactly the basic classes of $\G(f)$.  
\end{proof}

%\begin{remark}
%Note that Theorem \ref{thm:convMain}\ref{item:convex} and Lemma \ref{lem:classCap} imply that the converse of Theorem \ref{thm:graphCond} is true for all order-preserving, homogeneous, and multiplicatively convex functions on $\R^n_{>0}$.  
%\end{remark}

\color{black}
\section{Uniqueness of eigenvectors} \label{sec:unique}

Suppose that $f:\R^n_{>0} \rightarrow \R^n_{>0}$ is order-preserving and linear. If $f$ has two linearly independent eigenvectors $x, y \in \R^n_{>0}$, then every vector in the linear span of $\{x,y\}$ is also an eigenvector corresponding to the spectral radius. Since $\spn \{x,y\} \cap \R^n_{>0}$ is not bounded in Hilbert's projective metric, we see that the set of eigenvectors of $f$ in $\R^n_{>0}$ is nonempty and $d_H$-bounded if and only if $f$ has a unique eigenvector in $\R^n_{>0}$ up to scaling.  Of course, this is a well known part of linear Perron-Frobenius theory. 
Strikingly, it turns out that the exact same statement is true if $f:\R^n_{>0} \rightarrow \R^n_{>0}$ is order-preserving, homogeneous, and real analytic.

Let $U$ be an open subset of $\R^m$. A function $f: U \rightarrow \R^n$ is \emph{real analytic} if for every
$x \in U$, the function $f$ has a convergent power series in a neighborhood of $x$.

\begin{theorem} \label{thm:uniqueEig}
Let $f:\R^n_{>0} \rightarrow \R^n_{>0}$ be order-preserving, homogeneous, and real analytic. The eigenspace $E(f)$ is nonempty and bounded in $(\R^n_{>0}, d_H)$ if and only if $f$ has a unique eigenvector in $\R^n_{>0}$ up to scaling. 
\end{theorem}

In order to prove Theorem \ref{thm:uniqueEig}, we will prove a general result about the uniqueness of fixed points of nonexpansive maps on a real Banach space $X$. To do this, we need the following simple observation about Banach space norms. 

In what follows $X^*$ denotes the dual space of $X$ and $\|\cdot\|_*$ is the dual norm. For $x^* \in X^*$, we write $\inner{x,x^*}$ to represent $x^*(x)$.  The set valued \emph{duality map} $J: X \rightrightarrows X^*$ is given by 
$$J(x) := \{x^* \in X^* : \|x^*\|_* = \|x \| \text{ and } \inner{x,x^*} = \|x\| \|x^*\|_* \}.$$
The following is \cite[Lemma 3.3]{Lins21}. This observation is certainly not original to \cite{Lins21}, but we don't know of another reference.
\begin{lemma} \label{lem:subgradient}
Let $X$ be a real Banach space and let $x, y \in X$ with $\|x\|=1$.  Then there exists $x^* \in J(x)$ such that 
\begin{equation} \label{eq:subgradient}
\lim_{\mylambda \rightarrow \infty} \|\mylambda x - y\| - \inner{\mylambda x - y, x^*} = 0.
\end{equation}
\end{lemma}

A real Banach space $X$ has the \emph{fixed point property} if for every closed, bounded, convex subset $C \subset X$ and any nonexpansive map $f:C \rightarrow C$, $f$ has a fixed point in $C$.  All finite dimensional Banach spaces have the fixed point property, as do many infinite dimensional spaces.  

For functions on a real Banach space $X$, we will use a weak definition of real analytic which is sufficient for our purposes.  A map $f: X \rightarrow X$ is \emph{real analytic} if $t \mapsto \inner{f(x+ty),z^*}$ is a real analytic function for all $x,y \in X$, $z^* \in X^*$, and $t \in \R$. This weaker definition is equivalent to the more standard definition in terms of power series when $f$ is continuous and has Gateaux differentials of all orders \cite[Lemma 7.1 and Theorem 7.5]{BoSi71}. 

\begin{theorem} \label{thm:uniqueFixed}
Let $X$ be a real Banach space with the fixed point property.  Let $f:X \rightarrow X$ be nonexpansive and real analytic.  If $f$ has more than one fixed point, then the set of fixed points of $f$ is unbounded.\end{theorem}

\begin{proof}
We can assume without loss of generality that $f(0) = 0$. Otherwise, if $x_0$ is any fixed point of $f$, then replace $f$ by the function $x \mapsto f(x+x_0)-x_0$.  Now, suppose that $f$ also has a nonzero fixed point $w$.  Let $v = w/\|w\|$.  

For any $v^* \in J(v)$ and $0 < t < \|w\|$, observe that
$$\inner{f(tv),v^*} \le \|f(tv) - 0\| \le \|tv - 0\| = t$$
by nonexpansiveness.  Also 
\begin{align*}
\|w\|-\inner{f(tv),v^*} &= \inner{w-f(tv),v^*} \\
&\le \|w - f(tv)\| & (\text{since } \|v^*\|_* =1) \\
&\le \|w - tv\| & (\text{nonexpansiveness}) \\
&= \|w\| - t.
\end{align*}
Combining the two inequalities for $\inner{f(tv),v^*}$, we see that $\inner{f(tv),v^*} = t$ for all $0 < t < \|w\|$. Since $t \mapsto \inner{f(tv),v^*}-t$ is a real analytic function that is identically zero on the interval $0 \le t \le \|w\|$, it follows that 
\begin{equation} \label{tv}
\inner{f(tv),v^*} = t
\end{equation}
for all $t \in \R$ and $v^* \in J(v)$. 

Let $h(x) = \inf_{v^* \in J(v)} \inner{x,v^*}$ and 
let $H_\alpha = \{x \in X : h(x) \ge \alpha\}$ for all $\alpha \in \R$. Each $H_\alpha$ is closed and convex since it is the intersection of the closed half-spaces $\{x \in X : \inner{x,v^*} \ge \alpha\}$ where $v^* \in J(v)$. The sets $H_\alpha$ are also nonempty since they contain $tv$ for all $t \ge \alpha$. 

Choose any $x \in X$. Then 
\begin{align*}
\sup_{v^* \in J(v)} \inner{tv- f(x),v^*} &=  \sup_{v^* \in J(v)} \inner{f(tv)- f(x),v^*} & (\text{by } \eqref{tv}) \\ 
&\le \|f(tv) - f(x) \|  & (\text{since } \|v^*\|_* = 1) \\
& \le \|tv-x\|. & (\text{nonexpansiveness})
\end{align*}
By Lemma \ref{lem:subgradient}, there is an $x^* \in J(v)$ such that for every $\epsilon > 0$, 
$$\|tv-x\|  \le \inner{tv-x,x^*} + \epsilon$$
for all $t >0$ sufficiently large. Therefore   
$$\sup_{v^* \in J(v)} \inner{tv- f(x),v^*} \le \inner{tv - x,x^*} + \epsilon \le \sup_{v^* \in J(v)} \inner{tv-x,v^*} + \epsilon$$
when $t$ is large.  Equivalently, 
$$\inf_{v^* \in J(v)} \inner{x,v^*} \le \inf_{v^* \in J(v)} \inner{f(x),v^*} + \epsilon.$$
Since $\epsilon$ was arbitrary, we see that 
$$h(x) = \inf_{v^* \in J(v)} \inner{x,v^*} \le \inf_{v^* \in J(v)} \inner{f(x),v^*} = h(f(x)).$$
It follows that each $H_\alpha$ is invariant under $f$.
Since $f$ is nonexpansive and $f(0)=0$, the closed ball $B_R := \{x \in X : \|x\| \le R \}$ is also invariant under $f$ for all $R > 0$.  Therefore, the sets $H_\alpha \cap B_R$ are closed, convex, bounded, and invariant under $f$. They are also nonempty as long as $R \ge \alpha$.  Since $f$ has the fixed point property, we conclude that $H_\alpha \cap B_R$ must contain a fixed point of $f$ for every $\alpha > 0$ and every $R$ large enough.  Note that if $x \in H_\alpha$, then $\|x\| \ge \alpha$, so we see that $f$ has an unbounded set of fixed points.  
\end{proof}

\begin{remark}
The function $h$ in the proof of Theorem \ref{thm:uniqueFixed} is an example of a horofunction, and the sets $H_\alpha$ are horoballs on $X$.  Horofunctions have proven to be extremely useful tools for analyzing nonexpansive maps.  See e.g., \cite{Beardon90, GaVi12, Karlsson01, LLN16, Lins07, Lins21}.
\end{remark}

\begin{proof}[Proof of Theorem \ref{thm:uniqueEig}] 
Suppose that $f:\R^n_{>0} \rightarrow \R^n_{>0}$ is order-preserving, homogeneous, and real analytic. If $f$ has a unique eigenvector in $\R^n_{>0}$ up to scaling, then $E(f)$ is nonempty and bounded in $(\R^n_{>0}, d_H)$. To prove the converse, note that $d_H$ is a metric on the set $\Sigma = \{x \in \R^n_{>0} : x_n = 1\}$ by Proposition \ref{prop:dH}. The function $\hat{f}(x) = f(x)/(f(x)_n)$ maps $\Sigma$ into $\Sigma$ and is $d_H$-nonexpansive. Every eigenvector $x \in \R^n_{>0}$ of $f$ corresponds to a fixed point $\hat{x} = x/x_n$ of $\hat{f}$ in $\Sigma$, and every fixed point of $\hat{f}$ is an eigenvector of $f$. This means that $\hat{f}$ has a bounded set of fixed points in $(\Sigma,d_H)$.  

The entrywise natural logarithm $\log:\R^n_{>0} \rightarrow \R^n$ is an isometry from $(\Sigma,d_H)$ onto the subspace $V = \{x \in \R^n : x_n = 0\}$ with the variation norm 
\begin{equation} \label{varnorm}
\|x\|_\text{var} := \max_{i \in [n]} x_i - \min_{j \in [n]} x_j. 
\end{equation} 
The inverse of $\log$ is the entrywise natural exponential function $\exp:\R^n \rightarrow \R^n_{>0}$. The translated map $\exp \circ \hat{f} \circ \log$ is nonexpansive on $(V,\|\cdot\|_\text{var})$ and it is also a real analytic function with a bounded set of fixed points, so it must only have one fixed point $x \in V$ by Theorem \ref{thm:uniqueFixed}. Then $\exp(x)$ is the unique fixed point of $\hat{f}$ in $\Sigma$ and therefore all eigenvectors of $f$ in $\R^n_{>0}$ are multiples of $\exp(x)$.
\end{proof}

\section{Convergence of iterates to a unique eigenvector} \label{sec:iterates}

In some applications, it is important to know whether the iterates $f^k(x)$ will converge (after normalizing) to a single vector for all $x \in \R^n_{>0}$.  Of course, this can only happen if $f$ has a unique eigenvector in $\R^n_{>0}$ up to scaling. One known sufficient condition for the iterates to converge is if $f$ is differentiable at an eigenvector $u \in \R^n_{>0}$ and the derivative $f'(u)$ is primitive. In that case, $f^k(x)/\|f^k(x)\| \rightarrow u$ as $k \rightarrow \infty$ for all $x \in \R^n_{>0}$ \cite[Corollary 6.5.8]{LemmensNussbaum}. Here we suggest a somewhat more general sufficient condition. Recall that a strongly connected component of a directed graph is \emph{primitive} if there is an $m \in \N$ such that every ordered pair of vertices in the component can be connected by a path of length exactly $m$. 

\begin{theorem} \label{thm:convergence}
Let $f: \R^n_{>0} \rightarrow \R^n_{>0}$ be order-preserving and homogeneous.  Suppose that $f$ has an eigenvector $u \in \R^n_{>0}$ with $\|u\|= 1$ and $f$ is differentiable at $u$. If $\G(f'(u))$ has a unique final class, then $u$ is unique, that is, all other eigenvectors of $f$ in $\R^n_{>0}$ are scalar multiples of $u$. If, in addition, the final class is primitive, then $\lim_{k \rightarrow \infty} f^k(x)/\|f^k(x)\| = u$ for all $x \in \R^n_{>0}$. 
\end{theorem}
 
The condition that $\G(f'(u))$ has a unique final class that is primitive is sufficient but not necessary for the normalized iterates of $f$ to converge to the unique positive unit eigenvalue $u$. We shall see in the proof, however, that this condition is necessary for $r(f)$ to be a simple eigenvalue of the derivative $f'(u)$ and for all other eigenvalues of $f'(u)$ to have absolute value strictly smaller than $r(f)$.  This is an important consideration since it guarantees a linear rate of convergence for the iterates $f^k(x)/\|f^k(x)\|$. 

To prove Theorem \ref{thm:convergence} we use the following linear algebra proposition. Here $A = \begin{bmatrix} a_{ij} \end{bmatrix}_{i,j \in [n]}$ will be an $n$-by-$n$ nonnegative matrix.  We use $\rho(A)$ to denote the \emph{spectral radius} of $A$ which is the maximum of the absolute values of the eigenvalues of $A$. For a nonnegative $n$-by-$n$ matrix, the spectral radius $\rho(A)$ and the cone spectral radius $r(A)$ are the same (see e.g., \cite[Proposition 5.3.6 and Corollary 5.4.2]{LemmensNussbaum}). For any subset $J \subset [n]$, we let $A_J$ denote the principal submatrix $\begin{bmatrix}a_{ij} \end{bmatrix}_{i,j \in J}$.  We don't claim that this proposition is original, but since we are unaware of a reference, we include a proof.  

\begin{proposition} \label{prop:uniquefinal}
Let $A$ be an $n$-by-$n$ nonnegative matrix.  Then $A$ has a unique (up to scaling) positive eigenvector $u \in \R^n_{>0}$ if and only if $\G(A)$ has a unique final class $C$ and $\rho(A_C) > \rho(A_{[n] \bs C})$. If, in addition, the unique final class $C$ is primitive, then all other eigenvalues of $A$ have absolute value strictly less than $\rho(A)$.
\end{proposition}

\begin{proof}
In order for $A$ to have a positive eigenvector, or for $\G(A)$ to have a unique final class, $A$ must have at least one positive entry in each row.  Then $A$ maps $\R^n_{>0}$ into itself and $A$ is multiplicatively convex by Lemma \ref{lem:prop61}. For any $J \subseteq [n]$, $r(A^J_0) = \rho(A_J)$ and $r(A^{[n] \bs J}_0) = \rho(A_{[n] \bs J})$.  Therefore Theorem \ref{thm:convMain}\ref{item:convex} and Theorem \ref{thm:uniqueEig} imply that $A$ has a unique eigenvector in $\R^n_{>0}$ up to scaling if and only if $\G(A)$ has a unique final class $C$ and $\rho(A_{[n] \bs C}) < \rho(A_C)$.  

If the unique final class $C$ of $\G(A)$ is primitive, then $A_C$ is a primitive matrix, so $\rho(A) = \rho(A_C)$ is the only eigenvalue of $A_C$ with maximum modulus \cite[Theorem 8.5.3]{HornJohnson}.  Since $\rho(A_{[n] \bs C}) < \rho(A_C)$, it follows that $\rho(A)$ is also the only eigenvalue of $A$ with maximum modulus.  
\end{proof}

%\color{red}
%\begin{corollary} \label{cor:uniquefinal}
%Let $A$ be an $n$-by-$n$ nonnegative matrix such that $\G(A)$ has a unique final class.  If $A$ has an eigenvector $u \in \R^n_{>0}$, then $u$ is the unique eigenvector corresponding to $\rho(A)$ up to scaling. If, in addition, the unique final class is primitive, then all other eigenvalues of $A$ have absolute value strictly less than $\rho(A)$. \hl{This corollary is not important, so it should just be incorporated into the proof of the theorem.}
%\end{corollary}
%
%\begin{proof}
%Let $C$ denote the unique final class of $\G(A)$.  Note that $\rho(A_C) = r(A^C_0)$ and $\rho(A_{N \bs C}) = r(A^{N \bs C}_0)$ in the notation of Section \ref{sec:exist}. Since $A$ has an eigenvector in $\R^n_{>0}$, $A$ must be strongly nonnegative by Theorem \ref{thm:strongnonneg}.  Therefore $\rho(A_C) > \rho(A_{N \bs C})$. Then Proposition \ref{prop:uniquefinal} implies that $u$ is the unique eigenvector of $A$ corresponding to $\rho(A)$ up to scaling, and that all other eigenvalues of $A$ have absolute value strictly less than $\rho(A)$ if $C$ is primitive.
%\end{proof}
%\color{black}

\begin{proof}[Proof of Theorem \ref{thm:convergence}]
We can assume without loss of generality that $r(f) =1$ since we can replace $f$ by $r(f)^{-1}f$. Then all multiples of $u$ are fixed points of $f$. Since $f$ is order-preserving, $A := f'(u)$ is a nonnegative matrix; and since $f$ is homogeneous, $u$ is an eigenvector of $A$ with eigenvalue $r(f) = 1$. Since $u$ has all positive entries, this means that $\rho(A) = 1$.

Let $C$ denote the unique final class of $\G(A)$.  Note that $\rho(A_C) = r(A^C_0)$ and $\rho(A_{[n] \bs C}) = r(A^{[n] \bs C}_0)$ in the notation of Section \ref{sec:exist}. Since $A$ has an eigenvector in $\R^n_{>0}$, $A$ must be strongly nonnegative by Theorem \ref{thm:convMain} \ref{item:analyticConverse}.  Therefore $\rho(A_C) > \rho(A_{[n] \bs C})$. Then Proposition \ref{prop:uniquefinal} implies that $u$ is the unique eigenvector of $A$ corresponding to $\rho(A)$ up to scaling. By \cite[Corollary 6.4.7]{LemmensNussbaum}, $u$ is also the unique positive eigenvector of $f$ in $\R^n_{>0}$ with $\|u\| = 1$.

If $C$ is primitive, then Proposition \ref{prop:uniquefinal} also implies that all other eigenvalues of $A$ have absolute value strictly less than $\rho(A)$.
By the chain rule for Fr\'echet derivatives, $(f^k)'(u) = A^k$.  Therefore $u$ is the unique positive eigenvector of $(f^k)'(u)$ with $\|u \|=1$ for all $k \in \N$. By \cite[Corollary 6.4.7]{LemmensNussbaum}, $u$ is also the unique positive eigenvector of $f^k$ in $\R^n_{>0}$ with $\|u\| = 1$ for all $k \in \N$.   

For any $x \in \R^n_{>0}$, there are constants $\alpha, \beta > 0$ such that $\alpha u \le x \le \beta u$. Since $f$ is order-preserving, $\alpha u \le f^k(x) \le \beta u$ for all $k \in \N$.  By \cite[Theorem 8.1.7]{LemmensNussbaum} it follows that there is a period $p \in \N$ and a point $y \in \R^n_{\ge 0}$ such that $\lim_{k \rightarrow \infty} f^{kp}(x) = y$ and $y$ is periodic with period $p$ under iteration by $f$. Furthermore, $\alpha u \le y \le \beta u$, so $y \in \R^n_{>0}$.  But then $f^p(y) = y$, so $y$ must be a multiple of $u$ since $u$ is the only eigenvector of $f^p$ in $\R^n_{>0}$ up to scaling.  From this we see that $f^k(x)/\|f^k(x)\|$ converges to $u$ as $k \rightarrow \infty$. 
\end{proof}

Theorem \ref{thm:convergence} requires $f$ to be differentiable at the eigenvector $u \in \R^n_{>0}$. There are other conditions which can show that the normalized iterates of $f$ converge to an eigenvector when $f$ is not differentiable. See \cite[Section 6.5]{LemmensNussbaum} and \cite[Theorem 7.8]{AkGaNu14} for examples. Note that the result in \cite{AkGaNu14} can be used to prove a geometric rate of convergence of the iterates in Theorem \ref{thm:convergence} when the unique final class is primitive. Our next result shows that even when the iterates of $f$ do not converge to an eigenvector, the normalized iterates of $f+\id$ always do, as long as $f$ has an eigenvector in $\R^n_{>0}$.  A similar result, \cite[Theorem 11]{GaSt20}, proves that the normalized iterates of the function $x \mapsto x^{1/2} f(x)^{1/2}$ converge to an eigenvector of $f$ when $f:\R^n_{>0} \rightarrow \R^n_{>0}$ is order-preserving, homogeneous, and has an eigenvector in $\R^n_{>0}$. The result in \cite{GaSt20} is a multiplicative version of Krasnoselskii-Mann iteration. The observation about $f+\id$ turns out to be a special case of \cite[Theorem 2.3]{Jiang96} when $K$ is the standard cone $\R^n_{\ge 0}$, but the proof we give here works for any polyhedral cone. Recall that a \emph{polyhedral cone} is a closed cone that is a convex hull of a finite number of rays emanating from the origin in a Banach space. 
%Let $X$ be a finite dimensional real Banach space, let $\id$ denote the identity operator on $X$, and let $K \subset X$ be a polyhedral cone with nonempty interior. Suppose that $f: \inter K \rightarrow \inter K$ is order-preserving and homogeneous.  Let $\hat{f} = f(x)/\|f(x)\|$ for all $x \in \inter K$.
%Let $f:\R^n_{>0} \rightarrow \R^n_{>0}$ be order-preserving and homogeneous and let $\hat{f}(x) = f(x)/\|f(x)\|$ for every $x \in \R^n_{>0}$.  
%If $f$ has an eigenvector $u \in \inter K$ with $\|u \|=1$, then $u$ is a fixed point of $\hat{f}$.  Since $\hat{f}$ is nonexpansive with respect to $d_H$, it follows that the orbit of any $x \in \inter K$ under $\hat{f}$ is $d_H$-bounded. In this situation, it is known that there is a period $p$ (depending on $x$) such that $\lim_{k \rightarrow \infty} \hat{f}^{kp}(x) = y$ where $y$ is a periodic point of $\hat{f}$ with period $p$ \cite[Corollary 4.2.5 and Proposition 2.2.3]{LemmensNussbaum}.  So in general, the normalized iterates of $f$ might only converge to a periodic orbit rather than an eigenvector. However, the following result is true.
%the normalized iterates of $f+\id$ cannot converge to a periodic orbit with period $p >1$. 
%However, the following result shows that if we add a copy of the identity operator $\id$ to $f$, then the normalized iterates of the resulting map will converge to an eigenvector of $f$.  This can be useful to approximate an eigenvector in case one is known to exist.  

\begin{theorem} \label{thm:fplusI}
Let $K$ be a polyhedral cone with nonempty interior in a finite dimensional real Banach space $X$ and let $\id$ denote the identity operator on $X$.  Let $f: \inter K \rightarrow \inter K$ be order-preserving and homogeneous and let $g = f+\id$. If $f$ has an eigenvector in $\inter K$, then $g^k(x)/\|g^k(x)\|$ converges to an eigenvector of $f$ as $k \rightarrow \infty$ for every $x \in \R^n_{>0}$.  
\end{theorem}
 
Before proving Theorem \ref{thm:fplusI}, we briefly recall some facts about Hilbert's projective metric on cones other than the standard cone $\R^n_{\ge 0}$.  Let $K$ be a closed cone with nonempty interior in a finite dimensional real Banach space $X$ and let $K^*$ denote the dual cone $\{ x^* \in X^* : x^*(x) \ge 0 \text{ for all } x \in K \}$.  Then for any $x, y \in K \bs \{0\}$, 
$$d_H(x,y) = \sup_{\varphi, \psi \in K^* \bs \{ 0 \}} \log  \left( \frac{ \varphi(x) }{ \psi(x) } \, \frac{ \psi(y) }{ \varphi(y) } \right).$$
Note that if $x \in \inter K$, then $\varphi(x) > 0$ for all $\varphi \in K^* \bs \{0 \}$.  See \cite[Section 2.1]{LemmensNussbaum} for more details.

\begin{proof}[Proof of Theorem \ref{thm:fplusI}]
%%%\color{blue}
%%%Let $g(x) = \frac{f(x)+x}{\|f(x)+x\|}$ for every $x \in \inter K$.  Let $\Sigma = \{x \in \inter K : \|x\| = 1\}$.  Then $g:\Sigma \rightarrow \Sigma$ is nonexpansive with respect to $d_H$.  Suppose that $f$ has an eigenvector $u \in \Sigma$. Then $u$ is a fixed point of $g$.  For any other $x \in \Sigma$, the omega limit set $\omega(x)$ of $x$ under $g$ \hl{need to define this} is a nonempty compact subset of $\Sigma$. We may therefore choose elements $y, z \in \omega(x)$ such that $d_H(y,z)$ is maximal.  Then there exists $\phi, \psi \in K^*$ such that 
%%%$$d_H(y,z) = \log  \left( \frac{ \varphi(y) }{ \psi(y) } \, \frac{ \psi(z) }{ \varphi(z) } \right).$$
%%%Let 
%%%$$a = \frac{\psi(y)}{\phi(y)} \text{ and } b = \frac{\psi(z)}{\phi(z)}.$$
%%%Then $d_H(y,z) = \log(b/a)$. 
\color{black}
Let $\hat{g}(x) = g(x)/\|g(x)\|$ for all $x \in \inter K$.  We will prove that $\hat{g}$ has no periodic orbits in $\inter K$ with period $p > 1$.  Since $f$ and therefore $g$ both have an eigenvector in $\inter K$, it follows that the orbit $\hat{g}^k(x) = g^k(x)/\|g^k(x)\|$ is $d_H$-bounded. Then since $K$ is polyhedral, the iterates $\hat{g}^k(x)$ converge to a periodic orbit in $\inter K$ as $k \rightarrow \infty$ by \cite[Corollary 4.2.5 and Proposition 2.2.3]{LemmensNussbaum}. So proving that there are no nontrivial periodic orbits will guarantee that $\hat{g}^k(x)$ converges to an eigenvector of $g$.  And any eigenvector of $g$ is also an eigenvector of $f$.  

%If $f$ has an eigenvector $u \in \inter K$ with $\|u \|=1$, then $u$ is a fixed point of $\hat{f}$.  Since $\hat{f}$ is nonexpansive with respect to $d_H$, it follows that the orbit of any $x \in \inter K$ under $\hat{f}$ is $d_H$-bounded. In this situation, it is known that there is a period $p$ (depending on $x$) such that $\lim_{k \rightarrow \infty} \hat{f}^{kp}(x) = y$ where $y$ is a periodic point of $\hat{f}$ with period $p$ \cite[Corollary 4.2.5 and Proposition 2.2.3]{LemmensNussbaum}.  So in general, the normalized iterates of $f$ might only converge to a periodic orbit rather than an eigenvector. However, the following result is true.

Suppose by way of contradiction that $\hat{g}$ has a periodic orbit in $\inter K$ with period $p > 1$.  Chose some $y$ in this orbit.  Let $y^0 = y$ and $y^{k+1} = g(y^k)$ for each $k \in \N$.  Then $y^p$ is a scalar multiple of $y$.  Also, $y^1 = g(y) = y+f(y)$.  Then recursively, 
\begin{equation} \label{recursive} 
y^{k+1} = y^k+ f(y^k) = y + f(y^1) + \ldots + f(y^k)
\end{equation}
for all $k \in \N$.  

Choose $m \in \{1, \ldots, p-1\}$ such that $d_H(y,y^m)$ is maximal. Let $M = d_H(y,y^m)$. Since $g$ is nonexpansive, this implies that $d_H(y^k,y^j) \le M$ for all $k, j \in \N$.  

Since $d_H(y,y^m) = M$, there are linear functionals $\varphi, \psi$ in the dual cone $K^*$ such that 
$$d_H(y,y^m) = \log \left( \frac{ \varphi(y^m) }{ \psi(y^m) } \, \frac{ \psi(y) }{ \varphi(y) } \right) = M.$$
Then 
$$\frac{ \varphi(y^m) }{ \psi(y^m) }  = e^M \frac{ \varphi(y) }{ \psi(y) }.$$
Let $a = \varphi(y)/\psi(y)$ so $\varphi(y^m)/\psi(y^m) = a e^M$.
For any $k \in \N$, 
$$\log \left( \frac{\varphi(y^k)\psi(y)}{\psi(y^k) \varphi(y)} \right) \le d_H(y,y^k) \le M = \log \left( \frac{\varphi(y^m)\psi(y)}{\psi(y^m) \varphi(y)} \right)$$
which implies that 
$$\frac{\varphi(y^k)}{\psi(y^k)} \le \frac{\varphi(y^m)}{\psi(y^m)} = a e^M.$$
Similarly, 
$$\log \left( \frac{\varphi(y^m)\psi(y^k)}{\psi(y^m) \varphi(y^k)} \right) \le d_H(y^k,y^m) \le M = \log \left( \frac{\varphi(y^m)\psi(y)}{\psi(y^m) \varphi(y)} \right)$$
so
$$a = \frac{\varphi(y)}{\psi(y)} \le \frac{\varphi(y^k)}{\psi(y^k)}.$$
Therefore
$$a \le \frac{\varphi(y^k)}{\psi(y^k)} \le a e^M$$ 
for all $k \in \N$. 
\color{black}  %\hl{Roger points out that this step is confusing.  Maybe we should explain that the map $x \mapsto \log(\varphi(x)/\psi(x))$ is Lipschitz-1 from $(\inter K,d_H)$ to $\R$?}   

Let $b = \min \{ \varphi(f(y^k))/\psi(f(y^k)) : 0 \le k \le p-1\}$. Since $f$ is $d_H$-nonexpansive, $d_H(f(y^k),f(y^j)) \le M$ for all $k, j \in \N$.  Therefore 
$$b \le \frac{\varphi(f(y^k))}{\psi(f(y^k))} \le b e^M$$
for all $k \in \N$. Suppose that $b \le a$.  Then 
%$\varphi(f(y^k))/\psi(f(y^k)) \le a e^M$ for all $k \in \N$.  This means that 
$$\varphi(f(y^k)) - a e^M \psi(f(y^k)) \le \varphi(f(y^k)) - b e^M \psi(f(y^k)) \le 0$$
for all $k \in \N$.  At the same time, $\varphi(y) - a \psi(y) = 0$ so $\varphi(y) - a e^M \psi(y) < 0$.   Combined with \eqref{recursive}, we must have 
$\varphi(y^m) - a e^M \psi(y^m) < 0$, but that contradicts the fact that $\varphi(y^m)/\psi(y^m) = a e^M$.  

Now suppose that $b > a$. Then 
$$\varphi(f(y^k)) - a \psi(f(y^k)) > \varphi(f(y^k)) - b \psi(f(y^k)) \ge 0$$ 
for all $k \in \N$.  Using \eqref{recursive} we have $\varphi(y^p) - a\psi(y^p) > 0$.  But $y^p$ is a scalar multiple of $y$. So $\varphi(y^p)/\psi(y^p) = \varphi(y)/\psi(y) = a$ which is a contradiction.  We conclude that $\hat{g}$ cannot have a nontrivial periodic orbit, so $\hat{g}^k(x)$ converges to an eigenvector of $f$ for every $x \in \inter K$.  
\end{proof}

Since this paper was submitted, a more general version of Theorem \ref{thm:fplusI} has been published \cite[Theorem 3.2]{Lins23}. That paper also shows that the rate of convergence of the normalized iterates of $f+\id$ is at least linear in two important special cases: when $f$ is piecewise affine and when $f$ is real analytic and multiplicatively convex.

\section{Applications and examples} \label{sec:examples}

\subsection{Functions in the class $\mathcal{M}$} 

A large class of order-preserving, homogeneous, and real analytic functions $f:\R^n_{\ge 0} \rightarrow \R^n_{\ge 0}$ was introduced by Nussbaum in \cite{Nussbaum89}. The entries of these functions are built from various means of the entries of $x \in \R^n_{\ge 0}$. Let $r \in \R \bs \{0\}$ and let $\sigma \in \R^n_{\ge 0}$ be a probability vector, that is, $\sum_{i \in [n]} \sigma_i = 1$. The \emph{$(r,\sigma)$-mean} of $x \in \R^n_{\ge 0}$ is 
\begin{equation} \label{Mrsigma}
M_{r\sigma}(x) := \left( \sum_{i \in [n]} \sigma_i x_i^r \right)^{1/r}.
\end{equation}
When $r = 0$, we define 
\begin{equation} \label{Mrsigma0}
M_{r\sigma}(x) := \prod_{i \in \supp(\sigma)} x_i^{\sigma_i}.
\end{equation}
Note that if $r = 1, 0$, or $-1$, then $M_{r \sigma}(x)$ corresponds to a weighted arithmetic, geometric, or harmonic mean, respectively. Let $x,y \in \R^n_{>0}$ and $0 < \theta < 1$. When $r > 0$, H\"older's inequality implies that 
\begin{align*}
M_{r \sigma}(x^\theta y^{1-\theta}) &\le \left( \sum_{i \in [n]} (\sigma_i x_i^r)^\theta (\sigma_i y_i^r)^{1-\theta} \right)^{1/r} \\
&\le M_{r \sigma}(x)^{\theta} M_{r \sigma}(y)^{1-\theta}
\end{align*}
so $M_{r \sigma}$ is multiplicatively convex. When $r=0$, $ \log \circ M_{r \sigma} \circ \exp$ is linear so $M_{0 \sigma}$ is multiplicatively convex.  %For any $r \in \R$, $M_{r \sigma}$ is real analytic on $\R^n_{>0}$. 
\color{black}

Following \cite{Nussbaum89} (see also \cite[Section 6.6]{LemmensNussbaum}), we say that a function $f:\R^n_{\ge 0} \rightarrow \R^n_{>0}$ is in \emph{class} $M$ if each entry of $f$ is a positive linear combination of $(r,\sigma)$-means. We also say that $f$ is in \emph{class} $M_+$ (alternatively, $M_-$) if $f$ is in class $M$ and each $(r, \sigma)$-mean in the formula for $f$ has $r \ge 0$ ($r < 0$).   Then \emph{class} $\mathcal{M}$ is the smallest class of functions $f: \R^n_{\ge 0} \rightarrow \R^n_{\ge 0}$ that contains class $M$ and is closed under addition and composition. Likewise, \emph{class} $\mathcal{M}_+$ and $\mathcal{M}_-$ are the smallest classes containing $M_+$ and $M_-$, respectively, that are closed under addition and composition. Note that all functions in class $\mathcal{M}$ are real analytic on $\R^n_{>0}$.

By Lemma \ref{lem:prop61}, every $f \in \mathcal{M}_+$ is multiplicatively convex. Therefore we have the following result which combines Theorems \ref{thm:convMain} and \ref{thm:uniqueEig}.
%with Remark \ref{rem:strongernonneg} \hl{You can replace this remark reference with the new theorem at the end of section 3.}.  
\begin{theorem} \label{thm:Mplus}
Let $f: \R^n_{>0} \rightarrow \R^n_{>0}$ be in $\mathcal{M}_+$. Then $f$ has an eigenvector in $\R^n_{>0}$ if and only if $f$ is strongly nonnegative. If $f$ does have an eigenvector in $\R^n_{>0}$, then it is unique up to scaling if and only if $\G(f)$ has a unique final class $C$.  
\end{theorem}

Functions in class $\mathcal{M}_-$ are not multiplicatively convex so we do not have such simple conditions for existence and uniqueness of eigenvectors there. However, the results of Section \ref{sec:exist} and Section \ref{sec:unique} still let us find necessary and sufficient conditions on the symbolic parameters of any function in $\mathcal{M}$ to have a unique entrywise positive eigenvector up to scaling.  Our first detailed example is a function in $\mathcal{M}_-$ that comes from a population biology model.

\begin{example} \label{ex:Schoen}
The following order-preserving homogeneous function $f:\R^4_{>0} \rightarrow \R^4_{>0}$ was introduced by Schoen \cite{Schoen86} in a population model. The function was analyzed in detail in \cite[Section 3]{Nussbaum89}, and necessary and sufficient conditions were given there for the existence and uniqueness of an eigenvector in $\R^4_{>0}$. In what follows we will show how the methods described in this paper make the analysis easier and lead to the same conclusions. Let
$$f\left( \begin{bmatrix} x_1 \\ x_2 \\ x_3 \\ x_4 \end{bmatrix} \right) := 
\begin{bmatrix}
a_1 x_1 + b_1 \theta(x_1,x_2) + c_1 \theta(x_1,x_4) + d_1 \theta(x_2,x_3) \\
a_2 x_2 + b_2 \theta(x_1,x_2) + c_2 \theta(x_1,x_4) + d_2 \theta(x_2,x_3) \\
a_3 x_3 + b_3 \theta(x_3,x_4) + c_3 \theta(x_1,x_4) + d_3 \theta(x_2,x_3) \\
a_4 x_4 + b_4 \theta(x_3,x_4) + c_4 \theta(x_1,x_4) + d_4 \theta(x_2,x_3) 
\end{bmatrix}$$  
where $a_i, b_i, c_i, d_i$ are nonnegative constants, and $\theta(s,t) := (s^{-1}+t^{-1})^{-1}$. In the original model, the parameters $a_i$ were negative, but they can be made positive by adding a multiple of the identity to $f$ without affecting the existence of eigenvectors. Observe that $\theta(0,t) = 0$ and $\theta(\infty,t) = t$. We make the following additional assumptions about $f$:
\begin{itemize}
\item $a_i > 0$ for all $i$.
\item $d_1, c_2, c_3$ and $d_4$ are all positive. 
\item $a_1 < a_2+b_2$, $a_2<a_1+b_1$, $a_3<a_4+b_4$, and $a_4<a_3+b_3$. 
\end{itemize}

With these assumptions, we can find the hypergraphs $\Hm$ and $\Hp$ associated with $f$. Since each $a_i >0$, the hypergraph $\Hm$ has no hyperarcs. Figure \ref{fig:Schoen} shows the minimal hyperarcs of $\Hp$.  The only invariant sets in $\Hp$ are the singleton sets $\{1\}, \{2\}, \{3\}, \{4\}$, and the sets $\{1,2\}$, $\{1,3\}$, $\{2,4\}$, and $\{3,4\}$. 
All other nonempty $J \subsetneq [4]$ have $\reach(J,\Hp) = [4]$, so we do not need to check \eqref{superIllum} for those. It turns out that we won't need to check \eqref{superIllum} for all eight of the invariant sets either. If we can verify \eqref{superIllum} for each of the four invariant doubleton sets, then Theorem \ref{thm:quick} will show that \eqref{superIllum} also holds for $J = \{1\}, \{2\}, \{3\},$ and $\{4\}$, and thus $f$ has an eigenvector in $\R^4_{>0}$.  

\begin{figure}[ht] \label{fig:Schoen}
\begin{center}
\begin{tikzpicture}
\path (1,0) node[draw,shape=circle,scale=0.75] (A1) {1};
\path (2,0) node[draw,shape=circle,scale=0.75] (A4) {4};
\path (-2,0) node[draw,shape=circle,scale=0.75] (A2) {2};
\path (-1,0) node[draw,shape=circle,scale=0.75] (A3) {3};

\draw[thick,->] (A1) to [bend right=42] (A2);
\draw[thick,->] (A4) to [bend right=45] (A2);
\draw[thick,->] (A1) to [bend right=42] (A3);
\draw[thick,->] (A4) to [bend right=45] (A3);
\draw[thick,->] (A2) to [bend right=45] (A1);
\draw[thick,->] (A3) to [bend right=42] (A1);
\draw[thick,->] (A2) to [bend right=45] (A4);
\draw[thick,->] (A3) to [bend right=42] (A4);

\fill[left color=gray, right color=white,opacity=0.2] (A1) to [bend right=42] (A2) to [bend left=45] (A4) -- cycle;
\fill[left color=gray, right color=white,opacity=0.2] (A1) to [bend right=42] (A3) to [bend left=45] (A4) -- cycle;
\fill[left color=white, right color=gray,opacity=0.2] (A2) to [bend right=45] (A1) to [bend left=42] (A3) -- cycle;
\fill[left color=white, right color=gray,opacity=0.2] (A2) to [bend right=45] (A4) to [bend left=42] (A3) -- cycle;
\end{tikzpicture}
\end{center}
\caption{The hypergraph $\Hp$ for the maps in Example \ref{ex:Schoen}. }
\end{figure}
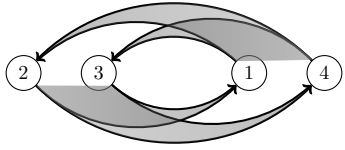

Below we write out \eqref{superIllum} with explicit formulas for $f^J_0$ and $f^{[4] \bs J}_\infty$ for each of the four invariant doubleton sets.  We will see that \eqref{superIllum} is always satisfied for two of them, and the remaining two amount to a necessary and sufficient condition on the parameters $a_i, b_i, c_i,$ and $d_i$ for the eigenspace $E(f)$ to be nonempty and bounded in $(\R^4_{>0},d_H)$.  Since $f$ is differentiable on $\R^4_{>0}$ and the derivative is always irreducible, it follows from Remark \ref{rem:irred} that any eigenvector of $f$ in $\R^4_{>0}$ must be unique up to scaling.  This means that the conditions below are necessary and sufficient for $f$ to have \emph{any} eigenvectors in $\R^4_{>0}$.  

\begin{description} 
\item[Case 1] $J = \{1,2\}$:

$$
r \left(
\begin{bmatrix}
a_1 x_1 + b_1 \theta(x_1,x_2)\\
a_2 x_2 + b_2 \theta(x_1,x_2)\\
0 \\ 
0
\end{bmatrix} \right) < 
\lambda \left(
\begin{bmatrix}
\infty \\
\infty \\
a_3 x_3 + b_3 \theta(x_3,x_4)+c_3 x_4 + d_3 x_3 \\
a_4 x_4 + b_4 \theta(x_3,x_4)+c_4 x_4 + d_4 x_3 
\end{bmatrix} \right).
$$

By applying Theorem \ref{thm:super} to the map 
$$\begin{bmatrix} x_1 \\ x_2 \end{bmatrix} \mapsto 
\begin{bmatrix}
a_1 x_1 + b_1 \theta(x_1,x_2)\\
a_2 x_2 + b_2 \theta(x_1,x_2)\\
\end{bmatrix}$$
we see that it has a nonempty and bounded eigenspace in $(\R^2_{>0},d_H)$ if and only if $a_1 < a_2+b_2$ and $a_2 < a_1 + b_1$. 
Since we have assumed that those inequalities are true, we can use Theorem \ref{thm:quick} to conclude that \eqref{superIllum} holds for $J = \{1\}$ and $J = \{2\}$ if it holds for $J = \{1,2\}$. Furthermore Theorem \ref{thm:super} can also be used to show that $f^{[4] \bs J}_\infty$ has an eigenvector in $\R^{\{3,4\}}_{>0}$. This implies that $\lambda(f^{[4] \bs J}_\infty)$ is equal to the cone spectral radius of the function
$$\begin{bmatrix} x_3 \\ x_4 \end{bmatrix} \mapsto 
\begin{bmatrix}
a_3 x_3 + b_3 \theta(x_1,x_2)+c_3 x_4 + d_3 x_3\\
a_4 x_4 + b_4 \theta(x_1,x_2)+c_4 x_4 + d_4 x_3\\
\end{bmatrix}.$$
It is not hard to calculate the spectral radii of these two functions.  Since they are both defined on $\R^2_{>0}$, one can explicitly solve for an eigenvector of the form $\begin{bmatrix} \tau \\ 1-\tau \end{bmatrix}$ where $0 < \tau <1$. Once the eigenvector is known, then the corresponding eigenvalue can be computed directly. This is done in \cite[Lemma 3.9]{Nussbaum89}. 

\item[Case 2] $J = \{1,3\}$:

$$
r \left(
\begin{bmatrix}
a_1 x_1 \\
0 \\
a_3 x_3 \\
0 \\
\end{bmatrix} \right) < \lambda \left(
\begin{bmatrix}
\infty \\
a_2 x_2 + b_2 x_2 + c_2 x_4 + d_2 x_2 \\
\infty \\
a_4 x_4 + b_4 x_4 + c_4 x_4 + d_4 x_2 
\end{bmatrix} \right).
$$
In this case, both $r(f^J_0)$ and $\lambda(f^{[4] \bs J}_\infty)$ are eigenvalues of linear maps which correspond to the 2-by-2 nonnegative matrices 
$$\begin{bmatrix} a_1 & 0 \\ 0 & a_3 \end{bmatrix} \text{ and } \begin{bmatrix} a_2+b_2+d_2 & c_2 \\ d_4 & a_4+b_4+c_4 \end{bmatrix}$$ 
respectively. Since we have assumed that $c_2, d_4 > 0$, the later matrix is irreducible and so has a unique positive eigenvector (up to scaling) with eigenvalue equal to $\lambda(f^{[4] \bs J}_\infty)$. It is easy to see that $r(f^J_0) = \max \{a_1,a_3\}$. Since we assumed that $a_1 < a_2+b_2$ and $a_3 < a_4+b_4$, we see that $r(f^J_0) < \lambda(f^{[4] \bs J}_\infty)$ is always true in this case.  

\item[Case 3] $J = \{2,4\}$:
$$r\left(
\begin{bmatrix}
0 \\
a_2 x_2 \\
0 \\
a_4 x_4 \\
\end{bmatrix} \right) < 
\lambda \left( 
\begin{bmatrix}
a_1 x_1 + b_1 x_1 + c_1 x_1 + d_1 x_3 \\
\infty \\
a_3 x_3 + b_3 x_3 + c_3 x_1 + d_3 x_3 \\
\infty \\
\end{bmatrix} \right).$$
Similar to Case 2, this condition is automatically true since $a_2 < a_1+b_1$ and $a_4 < a_3+b_3$.

\item[Case 4] $J = \{3,4\}$:
$$r\left(
\begin{bmatrix}
0 \\
0 \\
a_3 x_3 + b_3 \theta(x_3,x_4) \\
a_4 x_4 + b_4 \theta(x_3,x_4) 
\end{bmatrix} \right) < 
\lambda \left( 
\begin{bmatrix}
a_1 x_1 + b_1 \theta(x_1,x_2) + c_1 x_1 + d_1 x_2 \\
a_2 x_2 + b_2 \theta(x_1,x_2) + c_2 x_1 + d_2 x_2 \\
\infty \\
\infty \\
\end{bmatrix} \right).$$
Similar to Case 1, the function $f^J_0$ has a nonempty and $d_H$-bounded eigenspace in $\R^J_{>0}$ since $a_3 < a_4+b_4$ and $a_4 < a_3+b_3$.  By Theorem \ref{thm:quick}, if we verify the inequality above for $J = \{3,4\}$, then we will also verify \eqref{superIllum} for $\{3\}$ and $\{4\}$ as well.  The function $f^{[4] \bs J}_\infty$ also has an eigenvector in $\R^{[4] \bs J}_{>0}$.  So it is possible to explicitly compute both $r(f^{\{3,4\}}_0)$ and $\lambda(f^{\{1,2\}}_\infty)$ in order to check whether or not $r(f^{\{3,4\}}_0) < \lambda(f^{\{1,2\}}_\infty)$.
%%Using the function $\sigma$ to re-express this condition, we must have
%%\begin{equation} \label{Schoen2}
%%\sigma(a_3, b_3, 0, a_4, b_4, 0) < \sigma(a_1+c_1, b_1, d_1, a_2+c_2, b_2, d_2)
%%%r\left(
%%%\begin{bmatrix}
%%%a_3 x_3 + b_3 \theta(x_3,x_4) \\
%%%a_4 x_4 + b_4 \theta(x_3,x_4) 
%%%\end{bmatrix} \right) < 
%%%\lambda \left( 
%%%\begin{bmatrix}
%%%a_1 x_1 + b_1 \theta(x_1,x_2) + c_1 x_1 + d_1 x_2 \\
%%%a_2 x_2 + b_2 \theta(x_1,x_2) + c_2 x_1 + d_2 x_2 \\
%%%\end{bmatrix} \right)
%%\end{equation}
%in order for $f$ to have an eigenvector in $\R^4_{>0}$. 
\end{description}
Together the conditions $r(f^{\{1,2\}}_0) < \lambda(f^{\{3,4\}}_\infty)$ and $r(f^{\{3,4\}}_0) < \lambda(f^{\{1,2\}}_\infty)$ from Case 1 and Case 4 are necessary and sufficient for $f$ to have an eigenvector in $\R^4_{>0}$.  Note that these two inequalities are equivalent to the conditions given in \cite[Theorem 3.9]{Nussbaum89}.  The method used to find these conditions in \cite{Nussbaum89} involves finding all eigenvectors of $f$ corresponding to $r(f)$ on the boundary of $\R^4_{\ge 0}$ and then checking a condition on a Gateaux derivative at each one.  Our method is simpler and always results in a set of necessary and sufficient conditions for any order-preserving homogeneous function $f:\R^n_{>0} \rightarrow \R^n_{>0}$ to have a nonempty and $d_H$-bounded eigenspace.

\end{example}

\subsection{Nonnegative tensors} \label{sec:nonnegtensor}

An \emph{order-$d$ tensor} is an array of real numbers $\mathcal{A} = [\![ a_{j_1\cdots j_d} ]\!] \in \R^{n_1 \times \cdots \times n_d}$. Here we focus tensors with constant dimension $n_1 = \ldots = n_d = n$. The nonnegative eigenvalue problem for tensors seeks to find an eigenvalue $\lambda \in \R$ and an eigenvector $x \in \R^n$ such that 
$$\mathcal{A}x^{(d-1)} = \lambda x^{[d-1]}$$
where $x^{[d-1]}:= (x_1^{d-1}, \ldots, x_n^{d-1})$ and
$$(\mathcal{A}x^{(d-1)})_i := \sum_{1 \le j_2, \ldots, j_d \le n} a_{i j_2 \cdots j_d} x_{j_2} \cdots x_{j_n}.$$
This problem was independently introduced by Lim \cite{Lim05} and Qi \cite{Qi05}.  It is referred to by Qi as the H-eigenproblem to distinguish it from the Z-eigenproblem which seeks eigenpairs $(\lambda, x) \in \R \times \R^n$ such that 
$$\mathcal{A}x^{(d-1)} = \lambda x.$$

If the entries of $\mathcal{A}$ are all nonnegative, then we can rephrase the H-eigenproblem by letting 
$$f(x)_i = ((\mathcal{A}x^{(d-1)})_i)^{1/(d-1)} \text{ for } i \in [n].$$
Then $f: \R^n_{\ge 0} \rightarrow \R^n_{\ge 0}$ is order-preserving and homogeneous.  As long as there is a positive entry $a_{i j_2 \cdots j_d}$ in $\mathcal{A}$ for every $i \in [n]$, then $f$ maps $\R^n_{>0}$ into $\R^n_{>0}$ and we can ask whether or not $f$ has an eigenvector $x \in \R^n_{>0}$. If there is, then the eigenpair $(r(f)^{d-1},x)$ is a solution to the H-eigenproblem.  This problem and related generalizations have been studied in the context of general nonlinear Perron-Frobenius theory before, see e.g., \cite{AkGaHo20,ChPeZh08,FrGaHa13,GaTuHe19}.     

In \cite[Theorem 5]{HuQi16}, Hu and Qi give necessary and sufficient conditions for an order-$d$ nonnegative tensor $\mathcal{A} \in \R^{n \times \cdots \times n}$ to have an H-eigenvector in $\R^n_{>0}$. Their condition does not require or imply that the eigenspace is bounded in Hilbert's projective metric. Here we observe that the order-preserving homogeneous function $f$ associated with a nonnegative tensor $\mathcal{A}$ is in the class $\mathcal{M}_+$.  Therefore Theorem \ref{thm:Mplus} gives necessary and sufficient conditions for $f$ to have a unique eigenvector in $\R^n_{>0}$, up to scaling.  

%Here we note that the condition in Theorem \ref{thm:super} is necessary and sufficient for the set of H-eigenvectors to be nonempty and bounded in $(\R^n_{>0},d_H)$.  Since the order-preserving homogeneous maps $f$ associated with the H-eigenproblem always have real analytic entries on $\R^n_{>0}$, Theorem \ref{thm:super} also gives a necessary and sufficient condition for there to be a unique entrywise positive H-eigenvector up to scaling.  

% \hl{There is a lot to point out here: You should explain what this has to do with nonnegative tensors, how they are log-convex functions, and how that helps improve the hypergraph algorithm.}  

%\hl{QUESTION: The condition of Theorem 3.1 is always necessary and sufficient for existence and uniqueness of a positive eigenvector for the H-eigenvalue problem b/c those tensor maps are analytic. But does the convergence result always also apply?  Well, $\G(f'(u))$ always has a unique final class, since $u$ will be a unique eigenvector... but I don't think that final class has to be primitive? Hmm... even if iterates don't always converge, is the condition in section 5 necessary and sufficient for iterates to converge to a positive eigenvector?  }

\begin{example} \label{ex:tensor}
The following example corresponds to an order-3 nonnegative tensor in $\R^{4 \times 4 \times 4}$.  It is a slight variation of \cite[Example 5.4]{AkGaHo20}.  Unlike that example, the hypergraph condition of Theorem \ref{thm:AGH} fails for this map. It turns out that the existence of an eigenvector in $\R^n_{>0}$ will depend on the values of the parameters of the function $f$, and Theorem \ref{thm:Mplus} lets us give precise necessary and sufficient conditions for this map to have a unique (up to scaling) eigenvector in $\R^4_{>0}$. 
$$f\left( \begin{bmatrix} x_1 \\ x_2 \\ x_3 \\ x_4 \end{bmatrix} \right) := 
\begin{bmatrix}
\sqrt{a_1 x_1 x_2 + a_2 x_2^2} \\
\sqrt{b_1 x_1^2 + b_2 x_1 x_2 + b_3 x_2^2} \\
\sqrt{c_1 x_1^2 + c_2 x_1 x_2 + c_3 x_2 x_3} \\
\sqrt{d_1 x_1 x_4 + d_2 x_3^2 + d_3 x_4^2} \\
\end{bmatrix}.$$
%\hl{Because the map $\log \circ f \circ \exp$ is convex, the minimal hyperarcs of $\Hp$ are all regular arcs.  This makes it easy to determine the invariant sets for $\Hp$. Is that right???}

\begin{figure}[ht] \label{fig:Hpm2}
\begin{center}
\begin{tikzpicture}
\begin{scope}
\draw (-1.75,0.5) node {$\Hm$};
\path  (45:1) node[draw,shape=circle,scale=0.75] (A1) {1};
\path (135:1) node[draw,shape=circle,scale=0.75] (A2) {2};
\path (225:1) node[draw,shape=circle,scale=0.75] (A3) {3};
\path (315:1) node[draw,shape=circle,scale=0.75] (A4) {4};

\draw[thick,<-] (A1) to [bend right=30] (A2);
\draw[thick,->] (A1) to [bend right=20] (A3);
\draw[thick,->] (A2) to [bend left=30] (A3);

\fill[top color=white, bottom color=black,opacity=0.5] (A1) to [bend right=20] (A3) to [bend right=30] (A2);
\end{scope}

\begin{scope}[xshift=5cm]
\draw (-1.75,0.5) node {$\Hp$};
\path  (45:1) node[draw,shape=circle,scale=0.75] (B1) {1};
\path (135:1) node[draw,shape=circle,scale=0.75] (B2) {2};
\path (225:1) node[draw,shape=circle,scale=0.75] (B3) {3};
\path (315:1) node[draw,shape=circle,scale=0.75] (B4) {4};

\draw[thick,->] (B1) to [bend right=20] (B2);
\draw[thick,->] (B1) to [bend left=0] (B3);
\draw[thick,->] (B1) to [bend left=0] (B4);
\draw[thick,->] (B2) to [bend right=20] (B1);
\draw[thick,->] (B2) to [bend right=0] (B3);
\draw[thick,->] (B3) to [bend right=0] (B4);
\end{scope}

\end{tikzpicture}
\end{center}
\caption{The hypergraphs $\Hm$ and $\Hp$ for the maps in Example \ref{ex:tensor}. }
\end{figure}

The minimal hyperarcs of $\Hp$ and $\Hm$ are shown in Figure \ref{fig:Hpm2}. 
%It is interesting to note that the minimal hyperarcs of $\Hp$ are simple arcs. This always happens for order-preserving homogeneous functions $f:\R^n_{>0} \rightarrow \R^n_{>0}$ which have the property that $\log \circ f \circ \exp$ is a convex function. See \cite[]{} \hl{Actually, I don't think that this is ever well explained. Maybe it would be worth explaining here?  Maybe not? 
The only invariant subsets of $\Hp$ are $J = \{3,4\}$ and $J = \{4\}$.  Observe that $\reach(\{4\}^c,\Hm) = \{1,2,3\}$ and $\reach(\{3,4\}^c,\Hm) = \{1,2,3\}$, so both possible choices of $J$ fail the conditions in Corollary \ref{cor:AGH2}. 

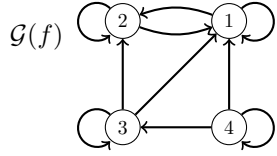
\begin{figure}[ht] \label{fig:G}
\begin{center}
\begin{tikzpicture}

\draw (-1.85,0.5) node {$\G(f)$};
\path  (45:1) node[draw,shape=circle,scale=0.75] (B1) {1};
\path (135:1) node[draw,shape=circle,scale=0.75] (B2) {2};
\path (225:1) node[draw,shape=circle,scale=0.75] (B3) {3};
\path (315:1) node[draw,shape=circle,scale=0.75] (B4) {4};

\draw[thick,->] (B1) to [bend right=20] (B2);
\draw[thick,->] (B3) to [bend left=0] (B1);
\draw[thick,->] (B4) to [bend left=0] (B1);
\draw[thick,->] (B2) to [bend right=20] (B1);
\draw[thick,->] (B2.135) arc(45:315:2.5mm);
\draw[thick,->] (B3.135) arc(45:315:2.5mm);
\draw[thick,->] (B3) to [bend right=0] (B2);
\draw[thick,->] (B4) to [bend right=0] (B3);
\draw[thick,->] (B4.45) arc(135:-135:2.5mm);
\draw[thick,->] (B1.45) arc(135:-135:2.5mm);

\end{tikzpicture}
\end{center}
\caption{The directed graph $\G(f)$ in Example \ref{ex:tensor}. }
\end{figure}

Since $f \in \mathcal{M}_+$, Theorem \ref{thm:Mplus} applies to $f$. The directed graph $\G(f)$ is shown in Figure \ref{fig:G}.  Observe that $\G(f)$ has a unique final class $C = \{1,2\}$.  Therefore $f$ has an eigenvector in $\R^4_{>0}$ (which is necessarily unique up to scaling) if and only if $f$ is strongly nonnegative, which by Lemma \ref{lem:classCap} is equivalent to $r(f^{\{3,4\}}_0) < r(f^{\{1,2\}}_0)$.  
Writing this condition down, we have

$$r \left( 
\begin{bmatrix}
0 \\
0 \\
0 \\
\sqrt{d_2 x_3^2 + d_3 x_4^2} \\
\end{bmatrix} \right) < r \left(
\begin{bmatrix}
\sqrt{a_1 x_1 x_2 + a_2 x_2^2} \\
\sqrt{b_1 x_1^2 + b_2 x_1 x_2 + b_3 x_2^2} \\
0 \\
0 \\
\end{bmatrix}\right).$$ 
It is easy to see that $r(f^{\{3,4\}}_0) = \sqrt{d_3}$, so the necessary and sufficient condition for $f$ to have an (unique) eigenvector in $\R^4_{>0}$ can be stated as follows
$$\sqrt{d_3}< r \left(
\begin{bmatrix}
\sqrt{a_1 x_1 x_2 + a_2 x_2^2} \\
\sqrt{b_1 x_1^2 + b_2 x_1 x_2 + b_3 x_2^2} \\
\end{bmatrix}\right).$$
\end{example}

\subsection{Topical functions and stochastic games}

The methods of this paper also apply to \emph{topical functions} which are order-preserving and additively homogeneous functions $T: \R^n \rightarrow \R^n$.  A function is \emph{additively homogeneous} if $T(x+c \one) = T(x) + c \one$ for all $x \in \R^n$ and $c \in \R$.  For a topical function $T$, we say that $x \in \R^n$ is an \emph{additive eigenvector} with \emph{eigenvalue} $\lambda \in \R$ if $T(x) = x+\lambda \one$.  If $T$ is topical, then the function $f = \exp \circ T \circ \log$ is order-preserving and homogeneous on $\R^n_{>0}$. In fact, the functions $\exp$ and $\log$ are order-preserving isometries between $(\R^n_{>0},d_H)$ and $(\R^n,\|\cdot\|_\text{var})$, were $\|\cdot\|_\text{var}$ is the variation norm defined in \eqref{varnorm}. Therefore $T$ is nonexpansive with respect to the variation norm. 
Furthermore, $T$ has an additive eigenvector $x \in \R^n$ if and only if $\exp(x)$ is an eigenvector of $f$ in $\R^n_{>0}$.  Thus all of the existence and uniqueness results for eigenvectors of order-preserving homogeneous functions on $\R^n_{>0}$ can be translated to corresponding results about additive eigenvectors for topical functions on $\R^n$. In some applications, it is more natural to work in the additive setting.  

Working with topical functions requires some adjustments to the notation and terminology of the previous sections.  
Let $T: \R^n \rightarrow \R^n$ be topical.  We define the upper and lower Collatz-Wielandt numbers of a topical function to be
$$r(T) := \inf_{x \in \R^n} \max_{i \in [n]} T(x)_i - x_i$$
and 
$$\lambda(T) := \sup_{x \in \R^n} \min_{i \in [n]} T(x)_i - x_i$$
respectively. There are iterative formulas analogous to \eqref{rlim} for both 
$$r(T) = \lim_{k \rightarrow \infty} \max_{i \in [n]} \frac{T^k(x)_i}{k}$$
and 
$$\lambda(T) = \lim_{k \rightarrow \infty} \min_{i \in [n]} \frac{T^k(x)_i}{k},$$
and both formulas hold for all $x \in \R^n$.

As with order-preserving homogeneous functions on $\R^n_{>0}$, topical functions extend continuously to $(-\infty,\infty]^n$ and to $[-\infty, \infty)^n$ \cite[Theorem 1]{BuSp00}, and the extensions are order-preserving and additively homogeneous \cite[Lemmas 6 \& 7]{BuSp00}.  We can extend the definition of the upper Collatz-Wielandt number $r(T)$ to order-preserving, additively homogeneous $T:[-\infty, \infty)^n \rightarrow [-\infty,\infty)^n$ and likewise, we can extend the definition of $\lambda(T)$ to order-preserving additively homogeneous $T:(-\infty, \infty]^n \rightarrow (-\infty, \infty]^n$.  For any subset $J \subset [n]$, we define $T^J_{-\infty} = P^J_{-\infty} T P^J_{-\infty}$ and $T^J_{\infty} = P^J_{\infty} T P^J_{\infty}$.  With this notation, we have the following restatement of Theorem \ref{thm:super}. 

\begin{theorem} \label{thm:topical}
Let $T: \R^n \rightarrow \R^n$ be topical.  The set of additive eigenvectors of $T$ is nonempty and bounded in the variation norm if and only if 
\begin{equation} \label{topicalCond}
r(T^J_{-\infty}) < \lambda(T^{[n] \bs J}_{\infty})
\end{equation}
for every nonempty proper subset $J \subset [n]$.  
\end{theorem} 

Following \cite{AkGaHo20}, we define two hypergraphs $\HmT$ and $\HpT$ for any topical function $T: \R^n \rightarrow \R^n$ as follows. The nodes of both hypergraphs are the set $[n]$ and
the hyperarcs of $\HmT$ (respectively, $\HpT$) are pairs $(I,\{j\})$ such that $I \subset [n]$, $j \in [n] \bs I$, and
$$\lim_{t \rightarrow -\infty} T(te_I)_j = -\infty \hspace*{1cm} (\text{resp.}, \lim_{t \rightarrow \infty} T(te_I)_j = \infty).$$
Note that $\HmT = \Hm$ and $\HpT = \Hp$ when $f = \exp \circ T \circ \log$.  Therefore each of the results in Subsection \ref{sec:hypergraph} can be translated directly to the additive setting.  In particular, we have the following corollary of Lemma \ref{lem:connect}.

\begin{lemma} \label{lem:connect2}
Let $T: \R^n \rightarrow \R^n$ be topical and let $J$ be a proper nonempty subset of $[n]$.  If $\reach(J^c,\HmT) = [n]$, then $r(T^J_{-\infty}) = -\infty$.  If $\reach(J,\HpT) = [n]$, then $r(T^{[n] \bs J}_\infty) = \infty$. In either case, \eqref{topicalCond} holds for $J$.  
\end{lemma}

%which are the Shapley operators associated with two-player zero sum stochastic games.  
Our next application is one where it is more natural to work in an additive setting.
A (finite) zero-sum \emph{stochastic game} consists of the following:
\begin{enumerate}
\item A finite state space, which we will assume is $[n]$,
\item two finite sets of actions $A$, $B$ (one for each player),
\item a payoff function $r:[n] \times A \times B \rightarrow \R$, and
\item a transition function $\rho$ which assigns a probability vector in $\R^n$ to each $(i,a,b) \in [n] \times A \times B$.
\end{enumerate}
At each stage of the game, Player A pays Player B a payoff $r(i,a,b)$ which depends on the state $i$ and the choices $a \in A$ and $b \in B$ of the two players during that stage. Then the state of the game transitions from $i$ to $j$ with probability equal to $\rho(i,a,b)_j$ for the next stage of the game. Naturally, player A wishes the payoffs at each stage to be as small as possible, while player B seeks the opposite. 

The \emph{Shapley operator} for a turn-based perfect-information stochastic game is a topical function $T: \R^n \rightarrow \R^n$ given by 
$$T(x)_i = \inf_{a \in A} \sup_{b \in B} \left( r(i,a,b) + \sum_{j = 1}^n x_j \rho(i,a,b)_j \right).$$
Shapley operators were originally introduced in \cite{Shapley53} to find the expected payoffs of repeated stochastic games where players discount future payoffs. Shortly after, Gillette \cite{Gillette57} studied the average payoff of repeated stochastic games without discounting. See e.g., \cite[Chapter 5]{Sorin} and the references therein for more details. 
Given a stochastic game in state $i$, the Shapley operator lets us recursively compute the expected total payoff from player A to player B over $k$ stages if both players are following their optimal strategies. The expected total payoff in that case is $T^k(0)_i$. If $T$ has an additive eigenvector, then the long run average payoff 
$$\lim_{k \rightarrow \infty} \frac{T^k(0)_i}{k}$$
does not depend on the initial state $i$, and is equal to $r(T)$.   

\begin{example} \label{ex:game}
Consider a simple two-player game with three states. Suppose that player B controls what happens in state 1, while player A controls states 2 and 3. In state 1, player B can choose to stay in state 1, in which case the payoff (from player A to B) is $r_1$.  Or player B can choose a different payoff $r_2$, but then the state will either remain in state 1 with probability $p_1$, or transition to state 2 with probability $1-p_1$ for the next stage. In state 2, player A chooses between payoffs $r_3$ and $r_4$. Choosing $r_3$ keeps the game in state 2, while $r_4$ sends the game to state 1 with probability $p_2$ or to state 3 with probability $1-p_2$.  Finally, in state 3, player A can choose between staying in state 3 for a payoff of $r_5$, or transitioning the game to state 1 for a payoff of $r_6$.  Figure \ref{fig:game} shows the states and transition choices facing the players.

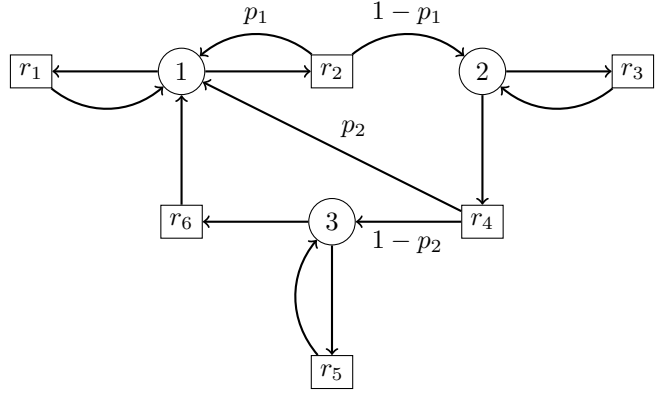
\begin{figure}[ht] \label{fig:game}
\begin{center}
\begin{tikzpicture}
\path (-4,2) node[draw,shape=rectangle,scale=1] (r1) {$r_1$};
\path (-2,2) node[draw,shape=circle,scale=1] (s1) {1};
\path  (0,2) node[draw,shape=rectangle,scale=1] (r2) {$r_2$};
\path  (2,2) node[draw,shape=circle,scale=1] (s2) {2};
\path  (4,2) node[draw,shape=rectangle,scale=1] (r3) {$r_3$};
\path (-2,0) node[draw,shape=rectangle,scale=1] (r6) {$r_6$};
\path (0,-2) node[draw,shape=rectangle,scale=1] (r5) {$r_5$};
\path (2,0) node[draw,shape=rectangle,scale=1] (r4) {$r_4$};
\path (0,0) node[draw,shape=circle,scale=1] (s3) {3};

\draw[thick,->] (s1) to (r1);
\draw[thick,->] (s1) to (r2);
\draw[thick,->] (s2) to (r3);
\draw[thick,->] (s2) to (r4);
\draw[thick,->] (s3) to (r5);
\draw[thick,->] (s3) to (r6);

\draw[thick,black,->] (r1) to [bend right=40] (s1);
\draw[thick,black,->] (r2) to [bend right=40] (s1);
\draw[thick,black,->] (r2) to [bend left=40] (s2);
\draw[thick,black,->] (r3) to [bend left=40] (s2);
\draw[thick,black,->] (r4) to (s3);
\draw[thick,black,->] (r4) to (s1);
\draw[thick,black,->] (r5) to [bend left=40] (s3);
\draw[thick,black,->] (r6) to (s1);

\draw (-1,2.5) node[above] {$p_1$};
\draw (1,2.5) node[above] {$1-p_1$};
\draw (0,1) node[above right] {$p_2$};
\draw (1,0) node[below] {$1-p_2$};
\end{tikzpicture}
\end{center}
\caption{The game board for Example \ref{ex:game}.}
\end{figure}

The Shapley operator for this game is 
$$T(x) := \begin{bmatrix} (r_1 + x_1) \vee (r_2 + p_1 x_1 + (1-p_1) x_2) \\
(r_3 + x_2) \wedge (r_4+p_2 x_1 + (1-p_2) x_3) \\
(r_5+ x_3) \wedge (r_6 + x_1) 
\end{bmatrix}$$
where $\vee$ is the max operation and $\wedge$ denotes min.  The minimal hyperarcs of the hypergraphs $\mathcal{H}^+_\infty(T)$ and $\mathcal{H}^-_\infty(T)$ are shown in Figure \ref{fig:HT}.

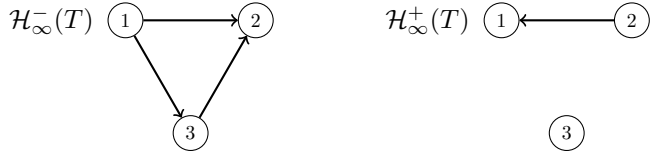
\begin{figure}[ht] \label{fig:HT}
\begin{center}
\begin{tikzpicture}
\begin{scope}
\draw (-1.85,0.5) node {$\mathcal{H}^-_\infty(T)$};
\path  (150:1) node[draw,shape=circle,scale=0.75] (A1) {1};
\path (30:1) node[draw,shape=circle,scale=0.75] (A2) {2};
\path (-90:1) node[draw,shape=circle,scale=0.75] (A3) {3};

\draw[thick,->] (A1) to (A2);
\draw[thick,->] (A1) to (A3);
\draw[thick,->] (A3) to (A2);
\end{scope}

\begin{scope}[xshift=5cm]
\draw (-1.85,0.5) node {$\mathcal{H}^+_\infty(T)$};
\path (150:1) node[draw,shape=circle,scale=0.75] (B1) {1};
\path (30:1) node[draw,shape=circle,scale=0.75] (B2) {2};
\path (-90:1) node[draw,shape=circle,scale=0.75] (B3) {3};

\draw[thick,->] (B2) to (B1);
\end{scope}

\end{tikzpicture}
\end{center}
\caption{The hypergraphs $\mathcal{H}^-_\infty(T)$ and $\mathcal{H}^+_\infty(T)$ for the Shapley operator in Example \ref{ex:game}. }
\end{figure}

Every proper subset $J \subset [3] = \{1,2,3\}$ either has $\reach(J^c,\HmT) = [3]$ or $\reach(J,\HpT)  = [3]$, except $J = \{1,3\}$ and $J = \{1\}$. Therefore, Theorem \ref{thm:topical} and Lemma \ref{lem:connect2} imply that $T$ has a nonempty and bounded set of additive eigenvectors in $(\R^n,\|\cdot\|_\text{var})$ if and only if the following two inequalities hold.  

\begin{description}
\item[Case 1] $J = \{1\}$
$$r_1 = r\left( \begin{bmatrix} 
r_1 + x_1 \\
-\infty \\
-\infty
\end{bmatrix} \right) < \lambda \left( \begin{bmatrix}
\infty \\
r_3 + x_2 \\
r_5 + x_3 
\end{bmatrix} \right) = \min \{r_3, r_5\}.$$

\item[Case 2] $J = \{1,3\}$
$$r_1 = r\left( \begin{bmatrix} 
r_1 + x_1 \\
-\infty \\
(r_5+x_3) \wedge (r_6+x_1)
\end{bmatrix} \right) < \lambda \left( \begin{bmatrix}
\infty \\
r_3 + x_2 \\
\infty
\end{bmatrix} \right) = r_3.$$
\end{description}

From these two cases, we see that $r_1 < \min \{r_3, r_5\}$ is necessary and sufficient for the set of additive eigenvectors of $T$ to be nonempty and bounded in the variation norm. 

\end{example}

\subsection*{Acknowledgement} The author wishes to thank Roger Nussbaum and St\'ephane Gaubert for their careful reading and helpful comments.  Thanks also to Bas Lemmens for his encouragement and suggestions.

\bibliography{DW2}
\bibliographystyle{plain}

\end{document}